\documentclass[12pt,a4paper]{scrartcl}
\usepackage[T1]{fontenc}
\usepackage{mathpazo}
\usepackage[english]{babel}	
\usepackage{csquotes}

\usepackage{graphicx}
\usepackage{color}
\definecolor{darkgreen}{rgb}{0.25,0.75,0.25}   
\definecolor{lightgrey}{rgb}{0.83,0.85,0.83}   
\definecolor{palegrey}{rgb}{0.93,0.95,0.93}
\usepackage{colortbl}
\newcommand{\cw}[1]{\colorbox{white}{#1}}
\newcommand{\cl}[1]{\colorbox{lightgrey}{#1}}
\newcommand{\cp}[1]{\colorbox{palegrey}{#1}}

\usepackage[a4paper]{geometry}
\usepackage{fullpage}
\usepackage{verbatim}
\usepackage{hyperref}
\newcommand{\phj}{\phantom{j\hspace{-1mm}}}
\usepackage{algorithm}

\usepackage[all]{xy}
\usepackage{mathtools}
\usepackage{array}
\usepackage{nicefrac}
\usepackage{setspace}
\usepackage{relsize}
\usepackage{tabto}
\usepackage{centernot}

\usepackage{amsmath}
\usepackage{amsthm}
\usepackage{amssymb}
\usepackage[mathscr]{eucal}

\theoremstyle{definition}
\newtheorem{defi}{Definition}[section]
\newtheorem*{rema}{Remark}	
\newtheorem{exam}{Example}

\theoremstyle{plain}
\newtheorem{theo}[defi]{Theorem}
\newtheorem{prop}[defi]{Proposition}
\newtheorem{coro}[defi]{Corollary}
\newtheorem{lemm}[defi]{Lemma}

\numberwithin{equation}{section}

\usepackage{tikz} 
\newcommand\around[1]{
  \,\tikz[anchor=-4ex,baseline]\node[draw,shape=rectangle,rounded corners,scale=0.5] {#1} ;}
\newcommand\Around[1]{
  \,\tikz[anchor=-6ex,baseline]\node[draw,shape=rectangle,rounded corners,scale=0.5] {#1} ;}

\newcommand{\refs}[1]{$\displaystyle{\ref{#1}}$\,}
\def\lb{\linebreak}

\newcommand{\mM}{\mathbb{M}}
\newcommand{\mN}{\mathbb{N}}
\newcommand{\mZ}{\mathbb{Z}}
\newcommand{\mP}{\mathbb{P}}
\newcommand{\mR}{\mathbb{R}}

\newcommand{\mS}{\mathbb{S}}

\newcommand{\msG}{\mathscr{G}}
\newcommand{\msA}{\mathscr{A}}
\newcommand{\msM}{\mathscr{M}}
\newcommand{\msI}{\mathscr{I}}
\newcommand{\msQ}{\mathscr{Q}}
\newcommand{\msT}{\mathscr{T}}

\newcommand{\msGS}{\msG_\mS}
\newcommand{\msQS}{\msQ_\mS}
\newcommand{\msTS}{\msT_\mS}
\newcommand{\msGSO}{\msG_{\mS1}}
\newcommand{\msTSO}{\msT_{\mS0}}

\newcommand{\msGm}{\msG^{(m)}}
\newcommand{\msQm}{\msQ^{(m)}}
\newcommand{\msTm}{\msT^{(m)}}
\newcommand{\msTmO}{\msT^{(m)}_0}

\newcommand{\mdot}{\!\cdot\!}
\newcommand{\ldot}{\,\cdot\,}

\newcommand{\mdiv}{\!\mid\!}
\newcommand{\copr}{\!\perp\!}

\DeclareMathOperator*{\argmax}{\arg\max\ }
\DeclareMathOperator*{\argmin}{\arg\min\ }
\DeclareMathOperator{\e}{e}
\DeclareMathOperator{\nd}{d}
\DeclareMathOperator{\ld}{ld}

\DeclareMathOperator{\md}{mod} 
\DeclareMathOperator{\congr}{\equiv}
\newcommand{\modo}[1]{\md #1}
\newcommand{\modu}[1]{\ (\md #1)}
\newcommand{\seq}[1]{(#1_i)}

\newcommand{\dSf}[2]{d_{(#1,\,#2)}}
\newcommand{\dSfxy}[4]{\dSf #1 #2(#3,#4)}
\newcommand{\dnSf}[2]{d^{\around{n}}_{(#1,\,#2)}}
\newcommand{\dnSfxy}[4]{\dnSf #1 #2(#3,#4)}

\newcommand{\ts}{\text{{\tiny $\text{$\square$}$}}}
\newcommand{\rts}{\frac{1}{\ts}}

\newcommand{\fp}[1]{#1^{(m,+)}}
\newcommand{\ft}[1]{#1^{(m,\ldot)}}
\newcommand{\fr}[1]{#1^{(m,\rts)}}
\newcommand{\fxp}[2]{\fp #1\bigl(\seq #2\bigr)}
\newcommand{\fxt}[2]{\ft #1\bigl(\seq #2\bigr)}
\newcommand{\fxr}[2]{\fr #1\bigl(\seq #2\bigr)}
\newcommand{\fxcp}[2]{\fp #1\bigl(\as #2\bigr)}
\newcommand{\fxct}[2]{\ft #1\bigl(\as #2\bigr)}
\newcommand{\fxcr}[2]{\fr #1\bigl(\as #2\bigr)}
\newcommand{\fmap}[1]{#1^{(m,+)}_{MA}}
\newcommand{\fmat}[1]{#1^{(m,\ldot)}_{MA}}
\newcommand{\fmar}[1]{#1^{(m,\rts)}_{MA}}

\newcommand{\dSfp}[2]{d^{(m,+)}_{(#1,\,#2)}}
\newcommand{\dSft}[2]{d^{(m,\ldot)}_{(#1,\,#2)}}
\newcommand{\dSfr}[2]{d^{(m,\rts)}_{(#1,\,#2)}}
\newcommand{\dSfxyp}[4]{\dSfp #1 #2\bigl(\seq #3,\seq #4\bigr)}
\newcommand{\dSfxyt}[4]{\dSft #1 #2\bigl(\seq #3,\seq #4\bigr)}
\newcommand{\dSfxyr}[4]{\dSfr #1 #2\bigl(\seq #3,\seq #4\bigr)}

\newcommand{\dSfxycp}[4]{\dSfp #1 #2\bigl(\as #3,\as #4\bigr)}
\newcommand{\dSfxyct}[4]{\dSft #1 #2\bigl(\as #3,\as #4\bigr)}
\newcommand{\dSfxycr}[4]{\dSfr #1 #2\bigl(\as #3,\as #4\bigr)}

\newcommand{\fn}[1]{#1^{\around{n}}}
\newcommand{\fnp}[1]{#1^{\around{n}(m,+)}}
\newcommand{\fnt}[1]{#1^{\around{n}(m,\ldot)}}
\newcommand{\fnr}[1]{#1^{\around{n}(m,\rts)}}

\newcommand{\dnSfp}[2]{d^{\around{n}(m,+)}_{(#1,\,#2)}}
\newcommand{\dnSft}[2]{d^{\around{n}(m,\ldot)}_{(#1,\,#2)}}
\newcommand{\dnSfr}[2]{d^{\around{n}(m,\rts)}_{(#1,\,#2)}}

\newcommand{\dnSfxycp}[4]{\dnSfp #1 #2\bigl(\as #3,\as #4\bigr)}
\newcommand{\dnSfxyct}[4]{\dnSft #1 #2\bigl(\as #3,\as #4\bigr)}
\newcommand{\dnSfxycr}[4]{\dnSfr #1 #2\bigl(\as #3,\as #4\bigr)}

\newcommand{\fnxcp}[2]{\fnp #1\bigl(\as #2\bigr)}
\newcommand{\fnxct}[2]{\fnt #1\bigl(\as #2\bigr)}
\newcommand{\fnxcr}[2]{\fnr #1\bigl(\as #2\bigr)}
\newcommand{\fnmap}[1]{#1^{\around{n}(m,+)}_{MA}}
\newcommand{\fnmat}[1]{#1^{\around{n}(m,\ldot)}_{MA}}
\newcommand{\fnmar}[1]{#1^{\around{n}(m,\rts)}_{MA}}

\DeclareMathAlphabet{\mathpzc}{OT1}{pzc}{m}{it}
\newcommand{\dd}[2]{d^{\around{#1}(m,\mathcal{O}_#2)}_{(\mathpzc{f}_#2,\mathcal{S}_#2)}}
\newcommand{\ff}[2]{{\mathpzc{f}_#2}^{\around{#1}(m,\mathcal{O}_#2)}_{MA}}

\newcommand{\as}[1]{\langle#1\rangle}
\newcommand{\ap}[2]{\langle#1\rangle_{#2}}

\begin{document}

\title{Pseudometrics and preorders\\on sets of integer sequences\\induced by arithmetic functions}
\author{Mario Ziller}
\date{}

\maketitle

\begin{abstract}
Starting from pseudometrics and preorders on sets of integers,\lb we extend the focus to sets of finite sequences of integers, in particular\lb sequences of consecutive integers. We outline existing concepts for\lb deriving centred pseudometrics and preorders in a given pseudometric space and their application to $\mathbb{Z}$ and develop approaches to generalize\lb the ideas to $\mathbb{Z}^m$. Sequences of consecutive integers represent a special case here and are examined in more detail.

Another main topic is the use of arithmetic functions in this context.\lb The types of pseudometrics and preorders examined in this paper\lb can be induced by suitable arithmetic functions. We derive fundamental conclusions about relationships between functions and preorders, as well as about equivalent and potentially distinct types of preorders.
\end{abstract}

\ \\
\tableofcontents


\section*{Introduction}
\addcontentsline{toc}{section}{\phj Introduction\phj}

Henceforth, we denote the set of integral numbers by $\mZ$ and the set of natural\lb numbers, i.e. positive integers, by $\mN$. $\mN_0=\mN\cup\{0\}$. $\mP=\{p_i : i\in\mN\}$ is the set of prime numbers with $p_1=2$. 

We abbreviate the set of positive representatives of the residue classes $\modo n$\lb as
$\mZ_n=\{x\in\mN:x\le n\}=\{1,\dots,n\}\simeq\mZ/n\mZ$. 

The set of real numbers is denoted by $\mR$, and $\mR_{\ge0}=\{x\in\mR: x\ge0 \}$.
\ \\\\[-1ex]

In a previous article, we investigated how preorders and partitions of sets of\lb integers can be generated using arithmetic functions \cite{Ziller_2023}. The main focus of the\lb current paper is on sets of finite integer sequences, in particular on sequences of\lb consecutive integers. We will generalise and expand the existent concepts for this purpose. As an introduction, we recapitulate the most important ideas.
\ \\\\[-1ex]

Starting from the consideration of the French railway metric, a simplified form of a centred pseudometric, sometimes called the British Rail metric, was derived that can be transferred to any pseudometric space \cite{Deza_Deza_2009, Ziller_2023}. Let $\mS$ be a space and $d$ a pseudometric on $\mS$. Then for every point $o\in\mS$ and arbitrary but different $x,y\in\mS$, the distance function $d_o(x,y) = d(o,x)+d(o,y)$ results in a centred pseudometric on $\mS$. Its centre is $o$. By definition, $d_o(x,y) = 0$ must be set for $x=y$.

On the other hand, a centred pseudometric can be generated on any set $\mS$\lb if a distance function with corresponding properties exists. Such a function we call generating function.

\begin{defi} \label{generating_functions}
Given a set $\mS$ and an element $o\in\mS$, every function $f:\mS\to\mR_{\ge0}$ with $f(o)=0$ is called generating. The set of all generating functions on $\mS$ is denoted as
\[\msGS=\{f:\mS\to\mR_{\ge0}\text{ \ with \ }f(o)=0\}.\]
\end{defi}

More precisely, we should use $\msG(\mS,o)$ or something similar for the set of\lb generating functions. In this section, we always consider the arbitrary but fixed\lb set $\mS$ and the centre $o$. Therefore, we write $\msGS$ for short. This also applies analogously to the supplementary sets defined below.

In all cases where the reference to the basic set and the centre is given by the context, we will also introduce shortened notations. This will be the case in all upcoming sections, each focusing on a specific underlying set with a characteristic centre.

\begin{lemm} \label{function_metric}
Let $x,y\in\mS$ and $f\in\msGS$. Then 
\begin{align*}
d(x,y)=\begin{cases}
   0 &\text{for $x=y$ \quad and}\\
   f(x)+f(y) &\text{for $x\ne y$} \end{cases}
\end{align*}
is an $o$-centred pseudometric on $\mS$. The function $f$ is given by $f(x)=d(o,x)$.
\end{lemm}
\begin{proof}
Identity and symmetry of $d$ follows by definition. The triangle inequality is satisfied because $f$ is non-negative. Let $z\in\mS$. Then
\[ d(x,y)+d(y,z)=f(x)+f(y)+f(y)+f(z)\ge f(x)+f(z)=d(x,z). \]
Furthermore, $d(o,x)=f(o)+f(x)=f(x)$.
\end{proof}

There are other non-negative functions on $\mS$ that can be easily transformed to\lb generating functions. We would like to point out two sets of them in particular.

\begin{lemm} \label{generating_supplements}
Let $o,x\in\mS$. Furthermore, we define the sets
\begin{align*}
\msQS &= \{g:\mS\to\mR_{\ge0}\text{ \ with \ }(g-1)\in\msGS\} \quad\text{and} \\
\msTS &= \{h:\mS\to\mR_{\ge0}\text{ \ with \ }(1-h)\in\msGS\},
\end{align*}
and call them supplements to $\msGS$. Then
\begin{align*}
g\in\msQS &\qquad\Longrightarrow\hspace{16.5mm} g(x)\ge1 \quad\land\quad g(o)=1\text{\qquad and} \\
h\in\msTS &\qquad\Longrightarrow\qquad 0\le h(x)\le1 \quad\land\quad h(o)=1.
\end{align*}
\end{lemm}
\begin{proof}
If $(g-1)\in\msGS$, then $g(x)-1\ge0$ or $g(x)\ge1$, and $g(o)-1=0$ or $g(o)=1$.
Analogously, if $(1-h)\in\msGS$ and $h(x)\ge0$, then $0\le h(x)\le1$, and $1-h(o)=0$ or $h(o)=1$. 
\end{proof}
\ 

Using functions of the sets $\msGS$ , $\msQS$, and $\msTS$, various pseudometrics can be constructed on $\mS$. We summarize the most important of them in the following proposition.

\begin{prop} \label{function_metric_other}
Let $x,y\in\mS$, $f\in\msGS$ , $g\in\msQS$, and $h\in\msTS$.

The following distance functions are $o$-centred pseudometrics on $\mS$ where $d_{\dots}(o,x)\in\msGS$.\lb In all cases, we set $d_{\dots}(x,y)=0$ \ for $x=y$.  For $x\ne y$, we define instead
\vspace*{-7mm}\\
\begin{align*}
\dSfxy f {{\msGS}} x y &= f(x) + f(y), \\
\dSfxy g {{\msQS}} x y &= g(x) + g(y) - 2, \text{ \quad and} \\
\dSfxy h {{\msTS}} x y &= 2 - h(x) - h(y).
\end{align*}
\end{prop}
\begin{proof}
The assertions follow immediately from Lemma~\refs{generating_supplements} because $(g-1)\in\msGS$ and $(1-h)\in\msGS$.
\end{proof}
\ \\[-2ex]

Every pseudometric space $\mS$ can be preordered by comparing the distances to an arbitrary but fixed point $o\in\mS$. In the case of centred pseudometrics, the centre itself is the natural choice of such a point. 

\begin{lemm} \label{induced_preorder}
Let $x,y\in\mS$, and $d$ be a pseudometric on $\mS$. For each point $o\in\mS$, the relation $\preceq_d$ defined by
\[x\preceq_d y \iff d(o,x)\le d(o,y)\]
is a preorder on $\mS$.
\end{lemm}
\begin{proof}
The specified relation is reflexive and transitive because the ordinary \enquote{is less than or equal to} relation is reflexive and transitive on $\mR_{\ge0}$.
\end{proof}

According to Proposition~\refs{function_metric_other}, we can express the corresponding preorders in terms of generating functions.

\begin{coro} \label{induced_preorder_other}
Let $x,y\in\mS$, $f\in\msGS$ , $g\in\msQS$, and $h\in\msTS$.

The following relations are preorders on $\mS$:
\begin{align*}
x\preceq_{\dSf f {\msGS}} y &\iff f(x) \le f(y), \\
x\preceq_{\dSf g {\msQS}} y &\iff g(x) \le g(y), \text{ \quad and} \\
x\preceq_{\dSf h {\msTS}} y &\iff h(x) \ge h(y).
\end{align*}
\end{coro}
\begin{proof}
This corollary is a consequence of Lemma~\refs{induced_preorder} and Proposition~\refs{function_metric_other} because 
\begin{align*}
\dSfxy g {{\msQS}} o x &= g(o)+g(x) - 2 = g(x) - 1 \text{ \quad and} \\
\dSfxy h {{\msTS}} o x &= 2 - h(o)-h(x) = 1 - h(x). \qedhere
\end{align*}
\end{proof}

For functions $h\in\msTS$, it is also possible to create pseudometrics using multiplicative\lb combinations of function values \cite{Ziller_2023}. We will not pursue this option here as these pseudometrics are not $o$-centred and will not actually result in additional preorders.
\ \\[1ex]

There are special relationships between the three sets $\msGS$, $\msQS$, $\msTS$ and some of their subsets. Firstly, we observe $\msQS\cup\msTS=\{f:\mS\to\mR_{\ge0}\text{ \ with \ }f(o)=1\}$ and\lb $\msQS\cap\msTS=\{f:\mS\to\{1\}\}$. The set of generating functions $\msGS$ has also two\lb matching subsets with $f(x)\ge1$ or $f(x)\le1$ for $x\ge1$, respectively. Using this\lb constellation, we define specific subsets of $\msGS$ and $\msTS$. Furthermore, we examine\lb several bijective functions between pairs of these sets that turn out to form\lb preorder-preserving isomorphisms.

\begin{defi} \label{G_subsets}
Let $x\in\mS$. We define
\begin{align*}
\msGSO &= \{f\in\msGS:f(x)\le1\} \quad\text{and} \\
\msTSO &= \{h\in\msTS:h(x)>0\}.
\end{align*}
\end{defi}

By this definition, Definition~\refs{generating_functions}, and Lemma~\refs{generating_supplements}, we explicitly get
\begin{align*}
\msGSO &= \{f:\mS\to\mR_{\ge0}\text{ \ with \ }0\le f(x)\le1\text{ and }f(o)=0\}, \\
\msTS &= \{h:\mS\to\mR_{\ge0}\text{ \ with \ }0\le h(x)\le1\text{ and }h(o)=1\},\quad\text{and} \\
\msTSO &= \{h:\mS\to\mR_{\ge0}\text{ \ with \ }0<h(x)\le1\text{ and }h(o)=1\}.
\end{align*}
\ \\[-2ex]

To begin with, there are canonical isomorphisms due to the linear transformations between the set of generating functions and their supplements.

\begin{prop} \label{isomorphic_linear}
Let $x,y\in\mS$. Then
\[ \msGS\quad\simeq\quad\msQS \qquad\text{and}\qquad \msGSO\quad\simeq\quad\msTS. \]

For each $f\in\msGS$, there exist $g\in\msQS$ and for each $f_1\in\msGSO$, there exist $h\in\msTS$ and vice versa, such that the induced preorders are pairwise equivalent on $\mS$, i.e.
\[ x\preceq_{\dSf f \msGS} y \iff x\preceq_{\dSf g \msQS} y \qquad\text{and}\qquad x\preceq_{\dSf {f_1} \msGSO} y \iff x\preceq_{\dSf h \msTS} y.\]
\end{prop}
\begin{proof}
The transformations $f=g-1$ and $f_1=1-h$ are bijective and strictly\lb monotonic. Thus,  $g=f+1$ and $h=1-f_1$. The equivalence of the preorders\lb follows by Corollary~\refs{induced_preorder_other}.
\end{proof}

However, if we exclude functions from $\msTS$ that have a function value $0$, then we have another isomorphism between $\msGS$ and $\msQS$, and isomorphisms between both of them and the reduced set $\msTSO$. There are triples of functions $f\in\msGS$, $g\in\msQS$, and $h\in\msTSO$ based on exponential-like transformations such that suitable induced preorders are also equivalent.

\begin{prop} \label{function_triple}
Let $x,y\in\mS$. Then
\[ \msGS\quad\simeq\quad\msQS\quad\simeq\quad\msTSO. \]

For each $f\in\msGS$, there exist $g\in\msQS$ and $h\in\msTSO$ and vice versa, such that the induced preorders are pairwise equivalent on $\mS$, i.e.
\[ x\preceq_{\dSf f \msGS} y \iff x\preceq_{\dSf g \msQS} y \iff x\preceq_{\dSf h \msTSO} y.\]
\end{prop}
\begin{proof}
For every $f\in\msGS$, we set $g=\e^{\ f}\in\msQS$ and $h=\e^{-f}\in\msTSO$. Both transformations are bijective because $\e^x$ is strictly monotonic. Thus, for every $g\in\msQS$, $f=\ln(g)\in\msGS$ and $h=\frac 1 g\in\msTSO$, and for every $h\in\msTSO$, $f=-\ln(h)\in\msGS$ and $g=\frac 1 h\in\msQS$.

 The equivalence of the preorders follows by Corollary~\refs{induced_preorder_other}.
\end{proof}

Consequently, there exists a preorder-preserving homomorphism on $\msQS$ and another isomorphism between $\msQS$ and $\msTSO$  which complete the corresponding isomorphisms\lb of the last two propositions.

\begin{coro} \label{homomorphic}
Let $x,y\in\mS$ and $g\in\msQS$. Then $\msQS\simeq\msQS$. There exists a homomorphism $\digamma:\msQS\to\msQS$ such that
\[ x\preceq_{\dSf g \msQS} y \iff x\preceq_{\dSf {\digamma(g)} {\msQS}} y. \]

Furthermore, again $\msQS\simeq\msTSO$. For each $g\in\msQS$, there exists another $h\in\msTSO$ and vice versa, such that the induced preorders are pairwise equivalent on $\mS$, i.e.
\[ x\preceq_{\dSf g \msQS} y \iff x\preceq_{\dSf h \msTSO} y.\]
\end{coro}
\begin{proof}
We set $\digamma(g)=\e^{\ g-1}\in\msQS$, so $g=\ln(\digamma(g))+1$. Furthermore, $h=\e^{\ 1-g}\in\msTSO$. Then $h=\frac{1}{\digamma(g)}$ and $g=1-\ln(h)=\ln(\digamma(g))+1$. The morphisms are bijective and strictly monotonic.

The asserted equivalence of the preorders follows from the Propositions~\refs{isomorphic_linear} and \refs{function_triple}. Based on them,
\begin{align*}
x\preceq_{\dSf {g-1} \msGS} y \iff x\preceq_{\dSf g \msQS} y\, &\qquad\text{and} \\
x\preceq_{\dSf {\ln(\digamma(g))} \msGS} y \hspace{34mm} &\iff x\preceq_{\dSf {\digamma(g)} {\msQS}} y \iff x\preceq_{\dSf {\frac{1}{\digamma(g)}} \msTSO} y. \qedhere
\end{align*}
\end{proof}
\ \\

We depict the described isomorphisms between the set of generating functions\lb and its supplements or related subsets in the following diagram.

\begin{figure}[H]
  \centering\vspace*{-3mm}
\begin{gather*}\xymatrixcolsep{2pc}\xymatrixrowsep{4pc}\xymatrix{
	\msGSO
		\ar@{..}[rrrrr]|-{\ \ \subseteq\ \ }
		\ar@{<->}[dd]_-{\ref{isomorphic_linear}}
	    & 
		    & & & & \msGS
				\ar@{<->}[d]_-{\ref{function_triple}}
				\ar@<5pt>@{<->}@/^1.5pc/[dd]^-{\ref{function_triple}} \\
		& \msQS
			\ar@{<->}[rrrru]^-(.56){\ref{isomorphic_linear}\,}\
			\ar@{<->}[rrrr]_-(.52){\ref{homomorphic}}\
			\ar@{<->}[rrrrd]_-(.592){\ref{homomorphic}\,\,}\
		    & & & & \msQS
				\ar@{<->}[d]_-{\ref{function_triple}} \\
	\msTS
		\ar@{..}[rrrrr]|-{\ \ \supseteq\ \ }
		& 
		    & & & & \msTSO
}\end{gather*}
   \caption{Preorder-preserving isomorphisms between the sets of generating and\\ supplementary functions on $\mS$.} \label{dia1}
\end{figure}
\ \\[-2ex]

In the starting section, we apply the definitions given above to sets of integers and explain how arithmetic functions can be used to derive pseudometrics and preorders. The introduction of a modulus $n\in\mN$ and the transformation $x\mapsto\bigl(1+ (x-1) \modo n\bigr)$ can enable extending the use of arithmetic functions to all integers \cite{Ziller_2023}. \!Such\, an\lb approach particularly highlights common divisibility properties of $x$ and $n$.

We will describe all relationships as generally as possible and limit ourselves to the essentials. This is intended to facilitate later application to other contexts. However, additive and multiplicative arithmetic functions are of particular interest in number theory. They have been extensively studied. Many of them can be transformed into generating functions or corresponding supplements. We will examine in more detail\lb specific subsets of additive and multiplicative arithmetic functions that have been\lb designated as admissible \cite{Ziller_2023}.
\ \\

In the following Section~\refs{MO_Z_M}, we extend the scope to integer sequences of a given length $m\in\mN$. They will be denoted by $\seq x _{i=0}^{m-1}$ where $x_i\in\mZ$. If the length of a sequence is obvious, we write $(x_i)$ for short. Therefore,
\[ \mZ^m = \{\seq x \mdiv x_i\in\mZ,\ i=0,\dots,m-1\} \]
is the set of all integer sequences of length $m$, i.e. the set of $m$-tuples of integral\lb numbers. 

We pursue the question of how to create pseudometrics and preorders on this set\lb so that some divisibility information of its elements can be gathered and compared.\lb In particular, the three Pythagorean means, namely the arithmetic mean (AM), the\lb geometric mean (GM), and the harmonic mean (HM), have proven suitable for this purpose in connection with generating functions and their supplements. We present, in a general and comprehensive manner, possibilities for extending the ideas described in Section~\refs{MO_Z} to integer sequences.
\ \\\\[-1ex]

Subsequently, we apply the provided concepts to sequences of consecutive integers. The finite sequence of $m$ consecutive integers $\seq x _{i=0}^{m-1} = (x+i)_{i=0}^{m-1}$, where $x=x_0$,\lb is denoted by $\ap x m$ \cite{Ziller_2020}. The first member of $\ap x m$ is $x$. So, the sequence $\ap x m$\lb is completely determined by $x$. In cases where the length is obvious, we write $\as x$ for short.

The set of all sequences of consecutive integers $\as x$ of length $m$ is symbolised by
\[\ap \mZ m=\{\as x : x\in\mZ\}.\]
By definition, $\ap \mZ m$ is a special subset of $\mZ^m$, $\ap \mZ m \subset \mZ^m$. Furthermore, there exists a bijection $\ap x m \mapsto x$ between $\ap \mZ m$ and $\mZ$, which will be explained in detail. The use of Pythagorean means leads to moving averages along a sequence of consecutive integers. This again underscores the existence of the aforementioned bijection.

At the end of section \refs{SCI}, we examine in particular once again pseudometrics and\lb preorders on $\ap \mZ m$ induced by admissible arithmetic functions. The underlying\lb assumptions ultimately result in various pseudometrics and preorders. However, some specific groups of preorders turned out to be equivalent. This leads to at most five generally distinct preorders.
\ \\\\[-1ex]

In the concluding remarks, we will discuss some details regarding the actual\lb number of different preorders that exist for a given function. Furthermore, the topic of extreme values of the corresponding moving averages is addressed. Several examples illustrate the derived statements.

\ \\\\


\section{Pseudometrics and preorders on $\mZ$} \label{MO_Z}

This section is essentially the summary of some results of a previous paper \cite{Ziller_2023}. We apply the definitions given in the introduction to sets of integers. In particular, we expound how arithmetic functions can be used to derive pseudometrics and preorders. Since $1$ is the only natural number dividing all integers, it seems obvious to choose $1$ as the centre of the considered pseudometrics. This enables the investigation of corresponding preorders and the focus on divisibility properties.

Properties of pseudometrics and preorders on a set are transferred to every subset of it. We describe the context of general relationships between centred pseudometrics and preorders on $\mZ$ as a whole. However, the results apply to arbitrary subsets of integers $\mM\subseteq\mZ$, even if the centre is not included, i.e. $1\notin\mM$.
 
Following Definition~\refs{generating_functions} and Lemma~\refs{generating_supplements}, we define the sets of generating functions and its supplements for all integers. These are the basic sets of functions that will be applied throughout the paper. We regard these functions as the basis functions of our context. Therefore, we completely refrain from using indices for them because the context remains clear.

\begin{defi} \label{generating_functions_Z}
The set of all generating functions of $1$-\hspace{0.25mm}centred pseudometrics on $\mZ$ results in
\[ \msG=\{f:\mZ\to\mR_{\ge0}\text{ \ with \ }f(1)=0\}. \]

The related supplementary sets are
\[ \msQ=\{g:\mZ\to\mR_{\ge0}\text{ \ with \ }(g-1)\in\msG\} \quad\text{and}\quad \msT=\{h:\mZ\to\mR_{\ge0}\text{ \ with \ }(1-h)\in\msG\}. \]
\end{defi}

Using functions of these sets $\msG$, $\msQ$, and $\msT$, various $1$-\hspace{0.25mm}centred pseudometrics and\lb corresponding preorders can be constructed on $\mZ$.

\begin{coro} \label{induced_pseudometric_preorder_Z}
Let $x,y\in\mZ$, $f\in\msG$ , $g\in\msQ$, and $h\in\msT$. For $x\ne y$, we define
\begin{align*}
\dSfxy f \msG x y &= f(x) + f(y), \\
\dSfxy g \msQ x y &= g(x) + g(y) - 2, \text{ \quad and} \\
\dSfxy h \msT x y &= 2 - h(x) - h(y).
\end{align*}
In all cases, we set $d_{\dots}(x,y)=0$ \ for $x=y$. Then, these distance functions are $1$-\hspace{0.25mm}centred pseudometrics on $\mZ$. \\

The following relations are preorders on $\mZ$.
\vspace*{-7mm}\\
\begin{align*}
x\preceq_{\dSf f \msG} y &\iff f(x) \le f(y), \\
x\preceq_{\dSf g \msQ} y &\iff g(x) \le g(y), \text{ \quad and} \\
x\preceq_{\dSf h \msT} y &\iff h(x) \ge h(y).
\end{align*}
\end{coro}
\begin{proof}
Confer Proposition~\refs{function_metric_other} and Corollary~\refs{induced_preorder_other}.
\end{proof}

If one defines $\msG_1 = \{f\in\msG:f(x)\le1\}$ and $\msT_0 = \{h\in\msT:h(x)>0\}$, there are compatible preorder-preserving isomorphisms as proven in Propositions~\refs{isomorphic_linear}, \refs{function_triple} and Corollary~\refs{homomorphic}. In what follows, we will limit ourselves to the exploration of isomorphisms based on exponential-like transformations, since they have a special relationship to arithmetic functions. According to Proposition~\refs{function_triple}, we conclude
\[ \msG\quad\simeq\quad\msQ\quad\simeq\quad\msT_0. \]
\ \\[-1ex]

Specific arithmetic functions can also be used as a starting point for creating\lb generating functions or their supplements. These, in turn, induce pseudometrics\lb and preorders on $\mZ$ as described above. There are several ways to modify suitable arithmetic functions such that their domains can be expanded to $\mZ$ \cite{Ziller_2023}. For this\lb purpose, we introduce an arbitrary but fixed modulus $n\in\mN>1$. Thus, we reveal\lb common divisibility properties of $n$ and the the function arguments. The number $1$ is divisible by no other natural number and the case $n=1$ can therefore be omitted.

The basic option to extend the domains of definition to $\mZ$, taking into account a modulus $n$, is to consider residue classes. To this end, we continue the projections of arithmetic functions onto the set $\mZ_n$ of positive representatives of the residue classes $\modo n$ to periodic functions. These, on the other hand, can be further continued to functions on to $\mZ$ as a whole. An arithmetic function $f$ is called periodic or $n$-periodic if $f(x+n)=f(x)$ for all $x\in\mN$ \cite{Schwarz_Spilker_1994}.

\begin{lemm} \label{AF_extend}
Let $n\in\mN>1$, $x\in\mZ$, and $f,g,h:\mN\to\mR_{\ge0}$ be non-negative arithmetic functions. We define
\[ \fn f(x)=f\bigl(1+ (x-1) \modo n\bigr). \]

If furthermore, $f(1)=0$, $g(1)=h(1)=1$, and $h(x)\le1\le g(x)$, then
\[ \fn f\in\msG, \quad\fn g\in\msQ, \text{\quad and \quad} \fn h\in\msT. \]
\end{lemm}
\begin{proof}
The transformation $x\mapsto\bigl(1+ (x-1) \modo n\bigr)$ maps $\mZ$ to $\mZ_n\subset\mN$ preserving the number $1$. Thus, $\fn f(1)=f(1)=0$ and $\fn g(1)=g(1)=\fn h(1)=h(1)=1$. The corresponding co-domains meet the requirements by definition.
\end{proof}
\ \\[-3ex]

Additive and multiplicative arithmetic functions are of particular interest in number theory. They have been extensively studied. Many of them can be transformed into generating functions or corresponding supplements. We will pay special attention to the following subsets of arithmetic functions which were called admissible \cite{Ziller_2023}.
\pagebreak

\begin{defi} \label{admissible}
Let $x\in\mN>1$. We define the following subsets of arithmetic functions and call its elements admissible.
\begin{align*}
\msA\,\ \ &\;:\ \!\text{ admissible additive arithmetic functions.} \\
\msA\,\ \ &=\{ f:\text{$f$ is an additive arithmetic function and }f(x)\ge0\}. \\[1.3ex]
\msM\ \ &\;:\ \!\text{ admissible multiplicative arithmetic functions, bounded below.} \\
\msM\ \ &=\{ g:\text{$g$ is a multiplicative arithmetic function and }g(x)\ge1\}. \\[1.3ex]
\msI\ \ \ &\;:\ \!\text{ admissible multiplicative arithmetic functions, bounded by a finite interval.} \\
\msI\ \ \ &=\{ h:\text{$h$ is a multiplicative arithmetic function and }0\le h(x)\le1\}. \\
\msI_0\ \,&=\{h\in\msI:0<h(x)\le1\}.		
\end{align*} 
\end{defi}

The sets $\msM$ and $\msI$ are not disjoint, similar to the general case of the supplements to the set of generating functions described in the introduction. We analogously get $\msM\cap\msI=\{f:\mN\to\{1\}\}$. The union of $\msM$ and $\msI$ includes all multiplicative arithmetic functions bounded below by $0$.

\begin{coro} \label{admissible_extend}
Let $n\in\mN>1$, $f\in\msA$, $g\in\msM$, $h\in\msI$, and $h_0\in\msI_0$. Then
\[ \fn f\in\msG, \quad\fn g\in\msQ, \quad \fn h\in\msT, \text{\quad and \quad} \fn h_0\in\msT_0. \]
\end{coro}
\begin{proof}
The functions $f$, $g$, and $h$ satisfy the assumptions of Lemma~\refs{AF_extend}.
\end{proof}

Analogous to Corollary~\refs{induced_pseudometric_preorder_Z}, we use the following notation for the corresponding pseudometrics and preorders on $\mZ$ that refer to $f\in\msA$, $g\in\msM$, and $h\in\msI$.
\begin{align*}
\dnSfxy f \msA x y &= f(x) + f(y), \\
\dnSfxy g \msM x y &= g(x) + g(y) - 2, \text{ \quad and} \\
\dnSfxy h \msI x y &= 2 - h(x) - h(y)
\end{align*}
for $\seq x\ne\seq y$. The corresponding preorders are
\begin{align*}
x\preceq_{\dnSf f \msA} y &\iff f(x) \le f(y), \\
x\preceq_{\dnSf g \msM} y &\iff g(x) \le g(y), \text{ \quad and} \\
x\preceq_{\dnSf h \msI} y &\iff h(x) \ge h(y).
\end{align*}
\ \\[1ex]

If we transfer the isomorphism between the sets $\msG$, $\msQ$ and $\msT_0$ analogously, we get a corresponding result for sets of admissible arithmetic functions.
\pagebreak

\begin{prop} \label{function_triple_Z}
Let $n\in\mN>1$, $x,y\in\mZ$, $f\in\msA$, $g\in\msM$, and $h\in\msI_0$. Then
\ \\[-1ex]
\[ \msA\quad\simeq\quad\msM\quad\simeq\quad\msI_0. \]
\ \\[-2ex]
For each $f\in\msA$, there exist $g\in\msM$ and $h\in\msI_0$ and vice versa, such that the following preorders are pairwise equivalent on $\mZ$, i.e.
\ \\[-1ex]
\[ x\preceq_{\dnSf f \msA} y \iff x\preceq_{\dnSf g \msM} y \iff x\preceq_{\dnSf h {\msI_0}}y. \]
\end{prop}
\begin{proof}
For every $f\in\msA$, \ $g=\e^{\ f}\in\msM$ and $h=\e^{-f}\in\msI_0$ are bijective transformations\lb\\[-2ex]
on $\mN$. Furthermore , we get
\ \\[-3ex]
\begin{align*}
\fn g(x) = g\bigl(1+ (x-1) \modo n\bigr) &= \e^{\ f\bigl(1+ (x-1) \modo n\bigr)} = \e^{\ \fn f(x)}, \qquad \text{and} \\
\fn h(x) = h\bigl(1+ (x-1) \modo n\bigr) &= \e^{-f\bigl(1+ (x-1) \modo n\bigr)} = \e^{-\fn f(x)}.
\end{align*}
The equivalence of the corresponding preorders follow by Corollaries~\refs{induced_pseudometric_preorder_Z} and \refs{admissible_extend}, since the exponential function is strictly monotonic.
\end{proof}
\ \\[-2ex]

In the following diagram, we depict the functions under consideration and their relationships to pseudometrics and associated preorders, as proven in the previous proposition. This turns out to be a special application of Proposition~\refs{function_triple} to $\mZ$ using admissible arithmetic functions.

\begin{figure}[H]
  \centering\vspace*{-3mm}
\begin{gather*}\xymatrixcolsep{2pc}\xymatrixrowsep{1pc}\xymatrix{
  f\in\msA
  	\ar[rr]^-(.58){\ref{induced_pseudometric_preorder_Z}}
	\ar@{<->}[d]_-{\ref{function_triple_Z}}
	& & \dnSf f \msA
  	    \ar[rr]^-{\ref{induced_pseudometric_preorder_Z}}
        & & \preceq_{\dnSf f \msA}
			& \ar@<-25pt>@{<->}@/^/[d]^-{\ref{function_triple_Z}} \\
  g=\e^{\ f}\in\msM
  	\ar[rr]^-(.47){\ref{induced_pseudometric_preorder_Z}}
  	\ar@{<->}[d]_-{\ref{function_triple_Z}}
  	& & \dnSf g \msM
  	    \ar[rr]^-{\ref{induced_pseudometric_preorder_Z}}
		& & \preceq_{\dnSf g \msM}
			& \ar@<-25pt>@{<->}@/^/[d]^-{\ref{function_triple_Z}} \\
  h=\e^{-f}\in\msI_0
  	\ar[rr]^-(.455){\ref{induced_pseudometric_preorder_Z}}
  	& & \dnSf h {\msI_0}
  	    \ar[rr]^-{\ref{induced_pseudometric_preorder_Z}}
		& & \preceq_{\dnSf h {\msI_0}}
			& \\
  h\in\msI\setminus\msI_0
  	\ar[rr]^-(.567){\ref{induced_pseudometric_preorder_Z}}
  	& & \dnSf h {\msI\setminus\msI_0}
  	    \ar[rr]^-{\ref{induced_pseudometric_preorder_Z}}
		& & \preceq_{\dnSf h {\msI\setminus\msI_0}}
			& \\
}\end{gather*}
\ \\[-2ex]
   \caption{Relationships between admissible arithmetic functions, induced\\ pseudometrics, and corresponding preorders on $\mZ$.} \label{dia2}
\end{figure}
\

We draw attention to the fact that functions $\fn f$ defined in Lemma~\refs{AF_extend} can be\lb neither additive nor multiplicative even if $f$ is an additive or a multiplicative\lb arithmetic function. This also applies
to the special case of n-even functions for which $f(x)=f\bigl(\gcd(x,n)\bigr)$ for all $x\in\mN$ \cite{McCarthy_1986, Schwarz_Spilker_1994}.

However, the additive or multiplicative property of an arithmetic function $f$ is transferred to $\fn f$ on the restricted domain $\mZ_n$, and further on all respective residue classes.

\begin{lemm} \label{add_mult}
Let  $n\in\mN>1$, $x,y\in\mZ$, and $f,k:\mN\to\mR_{\ge0}$.\\
If furthermore, $f$ is additive, $k$ is multiplicative, and\\
$\tilde{x}=\bigl(1+ (x-1) \modo n\bigr)\perp\tilde{y}=\bigl(1+ (y-1) \modo n\bigr)$ \quad with \quad $\tilde{x}\cdot\tilde{y}\in\mZ_n$,  \quad then
\ \\[-1ex]
\[ \fn f(x\cdot y) = \fn f(x) + \fn f(y) \text{\quad and\quad}\fn k(x\cdot y) = \fn k(x) \cdot \fn k(y). \]
\end{lemm}
\begin{proof}
\ \vspace*{-14.4mm}\\
\begin{align*}
\fn f(x\cdot y) &= f\bigl(1+ (x\cdot y-1) \modo n\bigr) = f(\tilde{x}\cdot\tilde{y}) \\
&= f(\tilde{x}) + f(\tilde{y}) = \fn f(x) + \fn f(y) \text{ \qquad and}\\
\fn k(x\cdot y) &= k\bigl(1+ (x\cdot y-1) \modo n\bigr) = k(\tilde{x}\cdot\tilde{y}) \\
&= k(\tilde{x}) \cdot k(\tilde{y}) = \fn k(x) \cdot \fn k(y). \qedhere
\end{align*}
\end{proof}

In other cases, additive and multiplicative properties may be violated. We give some examples.

\begin{exam}
We set $x=4$, $y=5$, and $n=13$. Then $x\copr y$.

\ \\[-1ex]
$f(x)=\Omega(x) =\Bigl|\bigl\{z=p^k : p\in\mP \land k\in\mN \land z\mdiv x\bigr\}\Bigr|$.

This function represents the number of different prime powers dividing n. The\lb\\[-5.25ex]

number of prime divisors (with repetition) is a totally additive function, $\Omega(x)\in\msA$.\\
Therefore, $f(4\cdot5)=f(4)+f(5)=2+1=3$.

$\fn f$ is not additive because

$\fn f(4\cdot5)=\fn f(20)=f(7)=1 \quad \ne \quad \fn f(4)+\fn f(5)=f(4)+f(5)=2+1=3$.

By the way, it is also not multiplicative:

$\fn f(4\cdot5)=\fn f(20)=f(7)=1 \quad \ne \quad \fn f(4)\cdot\fn f(5)=f(4)\cdot f(5)=2\cdot1=2$.

\ \\[-1ex]
$g(x)=\nd(x) = |\{d\in\mN : d\mdiv x\}|=\sum_{d\mid x} 1$.

The number of divisors is a multiplicative function, $\nd(x)\in\msM$.\\
Therefore, $g(4\cdot5)=g(4)\cdot g(5)=3\cdot2=6$.

$\fn g$ is not multiplicative because

$\fn g(4\cdot5)=\fn g(20)=g(7)=2 \quad \ne \quad \fn g(4)\cdot\fn g(5)=g(4)\cdot g(5)=3\cdot2=6$.

It is also not additive:

$\fn g(4\cdot5)=\fn g(20)=g(7)=2 \quad \ne \quad \fn g(4)+\fn g(5)=g(4)+g(5)=3+2=5$.

\ \\[-1ex]
$h(x)=\frac{1}{x}$.

The reciprocal function is totally multiplicative, $\frac{1}{x}\in\msI$.\\
Therefore, $h(4\cdot5)=h(4)\cdot h(5)=\frac{1}{4}\cdot\frac{1}{5}=\frac{1}{20}$.

$\fn h$ is not multiplicative because

$\fn h(4\cdot5)=\fn h(20)=h(7)=\frac{1}{7} \quad \ne \quad \fn h(4)\cdot\fn h(5)=h(4)\cdot h(5)=\frac{1}{4}\cdot\frac{1}{5}=\frac{1}{20}$.

It is also not additive:

$\fn h(4\cdot5)=\fn h(20)=h(7)=\frac{1}{7} \quad \ne \quad \fn h(4)+\fn h(5)=h(4)+h(5)=\frac{1}{4}+\frac{1}{5}=\frac{9}{20}$.
\end{exam}
\ \\[-3ex]

Finally, we point to a special subset of the considered generating functions $\fn f\in\msG$.\lb Let $n\in\mN>1$,  $x,y\in\mZ$, and $f:\mN\to\mR_{\ge0}$. It was previously proven that\lb
$d(x,y) = f\bigl(\gcd(x,n)\bigr)+f\bigl(\gcd(y,n)\bigr)$
for $x\ne y$, zero otherwise, is a pseudometric on $\mZ$ \cite{Ziller_2023}. 

The study of even functions $f(x)=f\bigl(\gcd(x,n)\bigr)$ is an interesting way to examine\lb common divisibility properties of $x$ and $n$. These functions are included in the\lb considered more general periodic functions $\fn f(x)=f\bigl(1+ (x-1) \modo n\bigr)$, which have been defined in Lemma~\refs{AF_extend}.

A function $\fn f(x)$ can have at most $n$ different function values for $x\in\mN\le n$. However, $f\bigl(\gcd(x,n)\bigr)$ can only have at most $\nd(n)\le n$ different values. If $x$ and $y$ have the same greatest common divisor with $n$, then the relating function values must be the same.

\begin{lemm} \label{mod_gcd}
Let  $n\in\mN>1$, $x,y\in\mZ$, and $f:\mN\to\mR_{\ge0}$.\\
If furthermore, $f(x)=f(y)$ for all $x,y\in\mN$ with $\gcd(x,n)=\gcd(y,n)$, then
\[ \fn f(x) = f\bigl(1+ (x-1) \modo n\bigr) = f\bigl(\gcd(x,n)\bigr) \quad\text{ for all } x\in\mZ. \]
\end{lemm}
\begin{proof}
We have $1\le\gcd(x,n)\le n$ and $\gcd(x,n)\congr x \modu n$.\\
Therefore, $\fn f(x) =  f\bigl(1+ (\gcd(x,n)-1) \modo n\bigr) = f\bigl(\gcd(x,n)\bigr)$.
\end{proof}

We would like to emphasize that the functions $\fn f$ are always periodic since they are defined by congruences $\modo n$. Consequently, this also applies to the functions\lb examined in Lemma~\refs{mod_gcd}. Such functions are also even functions and therefore\lb reversely periodic with the same period $n$.

\begin{lemm} \label{periodic_fn}
Let  $n\in\mN>1$, and $f:\mN\to\mR_{\ge0}$ be a non-negative arithmetic function.\\
The function $\fn f$ is then periodic with period $n$.
\end{lemm}
\begin{proof}
Let $x\in\mZ$. By Lemma~\refs{AF_extend}, we get
\[ \fn f(x) = f \bigl(1+ (x-1) \modo n\bigr)=f\bigl(1+ (x+n-1) \modo n\bigr)=\fn f(x+n). \qedhere \]
\end{proof}

\begin{lemm} \label{even_gcd}
Let  $n\in\mN>1$, $x,y\in\mZ$, and $f:\mN\to\mR_{\ge0}$. If furthermore $f(x)=f(y)$\lb for all $x,y\in\mN$ with $\gcd(x,n)=\gcd(y,n)$, then
\[ \fn f(x) = \fn f(-x) \quad\text{ and }\quad \fn f(x) = \fn f(n-x) \quad\text{ for all } x\in\mZ. \]
\end{lemm}
\begin{proof}
According to Lemma~\refs{mod_gcd}, we have
\[ \fn f(x) = f\bigl(\gcd(x,n)\bigr) = f\bigl(\gcd(-x,n)\bigr) = \fn f(-x). \]
And with Lemma~\refs{periodic_fn}, we get
\[ \fn f(x) = \fn f(-x) = \fn f(-x+n) = \fn f(n-x). \qedhere \]
\end{proof}
\begin{rema}
The properties of an even function and a periodic and reversely periodic function with the same period are equivalent:
\[ \fn f(x) = \fn f(-x) \quad\Longleftrightarrow\quad \fn f(x) = \fn f(n+x) = \fn f(n-x). \]
\end{rema}



\section{Pseudometrics and preorders on $\mZ^m$} \label{MO_Z_M}

This section aims to define pseudometrics and preorders on $\mZ^m$ based on the\lb principles explained above. Obvious ideas for metrics in Cartesian product spaces are, first of all, Minkowski distances ($L^p$), including the Euclidean distance ($L^2$)\lb and the Manhattan distance ($L^1$) \cite{Deza_Deza_2009}. Instead, we want to summarise the positional\lb distances like in the three Pythagorean means, the arithmetic mean ($AM$),\lb the geometric mean ($GM$), and the harmonic mean ($HM$). Arithmetic mean and\lb Manhattan distance are equivalent here.

\begin{lemm} \label{pseudometric_Z_m}
Let  $m\in\mN$ and $\seq x,\seq y\in\mZ^m$. For every pseudometric $d$ on $\mZ$,
\begin{align*}
d^{(m)}_{AM}\bigl(\seq x,\seq y\bigr) &= \frac{1}{m}\sum_{i=0}^{m-1}d(x_i,y_i), \\
d^{(m)}_{GM}\bigl(\seq x,\seq y\bigr) &= \left(\prod_{i=0}^{m-1}d(x_i,y_i)\right)^{\frac{1}{m}}\text{, \quad and} \\
d^{(m)}_{HM}\bigl(\seq x,\seq y\bigr) &= \begin{cases}
	\displaystyle{\frac{m}{\sum_{i=0}^{m-1}\frac{1}{d(x_i,y_i)}}} & \text{if } d(x_i,y_i)>0 \text{ for all pairs } (x_i,y_i), \\
	0 & \text{otherwise} \end{cases}
\end{align*}
are pseudometrics on $\mZ^m$.
\end{lemm}
\begin{proof}
 \ \\[1ex]
 Identity:

$\seq x=\seq y\iff \forall i=0,\dots,m-1 : x_i=y_i$.

Then also $d(x_i,y_i)=0$ for all $i=0,\dots,m-1$.
 \ \\[1.5ex]
Symmetry:

$d(x_i,y_i)=d(y_i,x_i)$ for all $i=0,\dots,m-1$.
 \ \\[1.5ex]
Triangle inequality:

We demonstrate that the relation \enquote{less than or equal to} in the set of non-negative real numbers transfers to the sum, the product, and the harmonised sum of reciprocals for the case $m=2$. The complete assertions follow by induction analogously.

Let now $z_0,z_1\in\mZ$, $0\le d(x_0,y_0)\le d(y_0,z_0)$, and $0\le d(x_1,y_1)\le d(y_1,z_1)$. Then, we directly get

$d(x_0,y_0)+d(x_1,y_1)\le d(y_0,z_0)+d(x_1,y_1)$\\
for the sum and

$d(x_0,y_0)\mdot d(x_1,y_1)\le d(y_0,z_0)\mdot d(x_1,y_1)\le d(y_0,z_0)\mdot d(y_1,z_1)$\\
for the product.

If $d(x_0,y_0)>0$ and $d(x_1,y_1)>0$, then also $d(y_0,z_0)>0$ and $d(y_1,z_1)>0$. Thus,
$\displaystyle{\frac{1}{d(x_0,y_0)}\ge\frac{1}{d(y_0,z_0)}}$\quad and\quad$\displaystyle{\frac{1}{d(x_1,y_1)}\ge\frac{1}{d(y_1,z_1)}}$, and finally

\begin{align*}
\frac{1}{d(x_0,y_0)}+\frac{1}{d(x_1,y_1)}&\ge\frac{1}{d(y_0,z_0)}+\frac{1}{d(x_1,y_1)}
\ge\frac{1}{d(y_0,z_0)}+\frac{1}{d(y_1,z_1)} \\
\iff\qquad\frac{1}{\displaystyle{\frac{1}{d(x_0,y_0)}+\frac{1}{d(x_1,y_1)}}}&\le\frac{1}{\displaystyle{\frac{1}{d(y_0,z_0)}+\frac{1}{d(y_1,z_1)}}}.
\end{align*}
Otherwise, if $d(x_0,y_0)=0$ or $d(x_1,y_1)=0$, then $d^{(2)}_{HM}\bigl(\seq x,\seq y\bigr)=0$ also satisfies the required inequality.
\end{proof}
\begin{rema}
If there is a pair $(x_i,y_i)$ with $x_i=y_i$ for any $0\le i\le m-1$, then\lb $d^{(m)}_{GM}\bigl(\seq x,\seq y\bigr)=d^{(m)}_{HM}\bigl(\seq x,\seq y\bigr)=0$. These distances are only interesting if\lb all $x_i\ne y_i$. This applies to different sequences of consecutive integers which we\lb consider in the next section. For the sake of completeness, it will also be discussed further here.
\end{rema}
\ \\[-2ex]

This lemma holds for arbitrary pseudometrics $d$ on $\mZ$. Thus, it also does for the\lb $1$-centred pseudometrics defined in the previous section. In the following, we examine only this subset of pseudometrics on $\mZ$. In all cases where $x=y$, we have $d(x,y)=0$.

Otherwise, we set $d(x,y)=d(1,x)+d(1,y)=f(x)+f(y)$, where $f\in\msG$.\lb According to Definition~\refs{generating_functions_Z} and Corollary~\refs{induced_pseudometric_preorder_Z}, we can also apply the supplementary\lb functions $g\in\msQ$ or $h\in\msT$ instead of $f$ such that $d(x,y)=g(x)+g(y)-2$ or\lb $d(x,y)=2-h(x)-h(y)$ for $x\ne y$, respectively.

For $x\ne y$, this results in the same pseudometrics and can be expressed as


\begin{align*}
\quad d^{(m)}_{AM}\bigl(\seq x,\seq y\bigr) &= \frac{1}{m}\sum_{i=0}^{m-1}d(x_i,y_i) = \frac{1}{m}\sum_{\substack{i=0\\x_i\ne y_i}}^{m-1}\bigl(f(x_i)+f(y_i)\bigr) \hspace*{100mm}\\
&= \frac{1}{m}\sum_{\substack{i=0\\x_i\ne y_i}}^{m-1}\bigl(g(x_i)+g(y_i)-2\bigr)
= \frac{1}{m}\sum_{\substack{i=0\\x_i\ne y_i}}^{m-1}\bigl(2-h(x_i)+h(y_i)\bigr), \\[2ex]
\quad d^{(m)}_{GM}\bigl(\seq x,\seq y\bigr) &= \left(\prod_{i=0}^{m-1}d(x_i,y_i)\right)^{\frac{1}{m}} = \left(\prod_{i=0}^{m-1} \bigl(f(x_i)+f(y_i)\bigr)\right)^{\frac{1}{m}} \hspace*{20mm}\\
&= \left(\prod_{i=0}^{m-1} \bigl(g(x_i)+g(y_i)-2\bigr)\right)^{\frac{1}{m}}
= \left(\prod_{i=0}^{m-1} \bigl(2-h(x_i)+h(y_i)\bigr)\right)^{\frac{1}{m}}\text{, \ \ and}
\end{align*}
\begin{align*}
\quad d^{(m)}_{HM}\bigl(\seq x,\seq y\bigr) &= \begin{cases}
	\displaystyle{\frac{m}{\sum_{i=0}^{m-1}\frac{1}{d(x_i,y_i)}}} & \text{if } d(x_i,y_i)>0 \text{ for all pairs } (x_i,y_i), \\
	0 & \text{otherwise} \end{cases} \\
&= \begin{cases}
	\displaystyle{\frac{m}{\sum_{i=0}^{m-1}\frac{1}{f(x_i)+f(y_i)}}} & \text{if } f(x_i)+f(y_i)>0 \text{ for all pairs } (x_i,y_i), \\
	0 & \text{otherwise} \end{cases} \\
&= \begin{cases}
	\displaystyle{\frac{m}{\sum_{i=0}^{m-1}\frac{1}{g(x_i)+g(y_i)-2}}} & \text{if } g(x_i)+g(y_i)>2 \text{ for all pairs } (x_i,y_i), \\
	0 & \text{otherwise} \end{cases} \\
&= \begin{cases}
	\displaystyle{\frac{m}{\sum_{i=0}^{m-1}\frac{1}{2-h(x_i)+h(y_i)}}} & \text{if } h(x_i)+h(y_i)<2 \text{ for all pairs } (x_i,y_i), \\
	0 & \text{otherwise}. \end{cases} \hspace*{20mm}\\
\end{align*}

Together with each of these distances, $\mZ^m$ forms a pseudometric space. Thus, we\lb investigate the question of which centred pseudometrics and generating functions based on the principles explained above can be derived from these spaces. The\lb sequence $(1)\in\mZ^m$ seems to be the natural choice here. It contains no information about divisibility because $1$ is not divisible by any prime number. It is a kind of\lb expression of indivisibility. The distance from this point therefore provides a measure of divisibility.

Under this premise, we redefine generating functions and its supplements for the current context in general. We use the superscript $^{(m)}$ as a short tag for the reference to $\mZ^m$. The special case of $m=1$ leads to $\mZ$ as mentioned above. The corresponding index $^{(1)}$ is omitted here in accordance with Definition~\refs{generating_functions_Z}. Evidently, $\mZ$ is a special case of $\mZ^m$ for $m=1$. We would like to point out that in this case all statements made in the following section correspond to those in the previous section.

\begin{defi} \label{generating_functions_Z_m}
The set of generating functions of $(1)$-\hspace{0.25mm}centred pseudometrics on $\mZ^m$ is now
\[ \msGm = \left\{ f^{(m)}:\mZ^m\to\mR_{\ge0}\text{ \ with \ }f^{(m)}\bigl((1)\bigr)=0 \right\}. \]

The related supplements are
\begin{align*}
\msQm &= \left\{ g^{(m)}:\mZ^m\to\mR_{\ge0}\text{ \ with \ }(g^{(m)}-1)\in\msGm\right\} \quad\text{and} \\ \msTm &= \left\{ h^{(m)}:\mZ^m\to\mR_{\ge0}\text{ \ with \ }(1-h^{(m)})\in\msGm\right\}.
\end{align*}
\end{defi}

Lemma~\refs{pseudometric_Z_m} implies three kinds of generating functions and the related pseudo-\lb metrics. The results are again Pythagorean means of simple basis function values. Where appropriate, we supplement the superscript index with \enquote{$+$}, \enquote{$\ldot$} or \enquote{$\rts$} to\lb indicate the combination of the values of the basis functions by arithmetic, geometric or harmonic mean, respectively.

\begin{lemm} \label{implied_generating_functions_Z_m}
Let $m\in\mN$, $x,y\in\mZ$, and $\seq x\in\mZ^m$. Furthermore, let $d(x,y)$ be a pseudo-\lb metric on $\mZ$. With $f(x)=d(1,x)$, \\[-2ex]
\begin{align*}
\fxp f x &= d^{(m)}_{AM}\bigl((1),\seq x\bigr) = \frac{1}{m}\sum_{i=0}^{m-1}f\seq x, \\
\fxt f x &= d^{(m)}_{GM}\bigl((1),\seq x\bigr) = \left(\prod_{i=0}^{m-1}f\seq x\right)^{\frac{1}{m}}\text{, \quad and} \\
\fxr f x &= d^{(m)}_{HM}\bigl((1),\seq x\bigr) = \begin{cases}
	\displaystyle{\frac{m}{\sum_{i=0}^{m-1}\frac{1}{f\seq x}}} & \text{if } f\seq x>0 \text{ for all } x_i, \\
	0 & \text{otherwise} \end{cases}
\end{align*}
are generating functions on $\mZ^m$, i.e. $ \fp f,\ \ft f,\ \fr f \in \msGm$.

Let furthermore $g\in\msQ$ and $h\in\msT$. With the explicit definitions
\begin{align*}
\fxp g x &= \frac{1}{m}\sum_{i=0}^{m-1}g\seq x, \\
\fxt g x &= \left(\prod_{i=0}^{m-1}g\seq x\right)^{\frac{1}{m}}, \\
\fxr g x &= \displaystyle{\frac{m}{\sum_{i=0}^{m-1}\frac{1}{g\seq x}}}, \\
\fxp h x &= \frac{1}{m}\sum_{i=0}^{m-1}h\seq x, \\
\fxt h x &= \left(\prod_{i=0}^{m-1}h\seq x\right)^{\frac{1}{m}}\text{, \quad and} \\
\fxr h x &= \begin{cases}
	\displaystyle{\frac{m}{\sum_{i=0}^{m-1}\frac{1}{h\seq x}}} & \text{if } h\seq x>0 \text{ for all } x_i, \\
	0 & \text{otherwise,} \end{cases}
\end{align*}
these functions are supplementary functions on $\mZ^m$ where $\fp g,\ \ft g,\ \fr g \in \msQm$\lb and $ \fp h,\ \ft h,\ \fr h \in \msTm$.
\end{lemm}
\begin{proof}
By definition, $f\in\msG$ and the co-domains of the functions $\fp f$, $\ft f$,\lb and $\fr f$  are each $\mR_{\ge0}$. With $f(1)=d(1,1)=0$, we finally get
\[ \fp f\bigl((1)\bigr) = \ft f\bigl((1)\bigr) = \fr f\bigl((1)\bigr) = 0. \]

We note that by definition, $g\seq x>0$ always holds. A case differentiation in the definition is therefore not necessary. From $g(1)=h(1)=1$, we conclude

\begin{align*}
\fp g\bigl((1)\bigr) &= \ft g\bigl((1)\bigr) = \fr g\bigl((1)\bigr) = 1 \quad\text{and} \\
\fp h\bigl((1)\bigr) &= \ft h\bigl((1)\bigr) = \fr h\bigl((1)\bigr) = 1.
\end{align*}
The remaining assumptions follow from Definition~\refs{generating_functions_Z_m}.
\end{proof}
\ \\[-1ex]

From $\msGm$, $\msQm$ and $\msTm$, pseudometrics and preorders on $\mZ^m$ can be derived\lb analogously, as was done in Corollary~\refs{induced_pseudometric_preorder_Z}. We continue to pursue only the specific functions defined in the previous lemma.

\begin{coro} \label{induced_pseudometric_preorder_Z_m}
Let $\seq x,\ \seq y\in\mZ^m$, $f\in\msG$, $g\in\msQ$, and $h\in\msT$. For $\seq x\ne \seq y$, we set
\begin{align*}
\dSfxyp f \msG x y &= \fxp f x + \fxp f y, \\
\dSfxyt f \msG x y &= \fxt f x + \fxt f y, \\
\dSfxyr f \msG x y &= \fxr f x + \fxr f y, \\
\dSfxyp g \msQ x y &= \fxp g x + \fxp g y - 2, \\
\dSfxyt g \msQ x y &= \fxt g x + \fxt g y - 2, \\
\dSfxyr g \msQ x y &= \fxr g x + \fxr g y - 2, \\
\dSfxyp h \msT x y &= 2 - \fxp h x - \fxp h y), \\
\dSfxyt h \msT x y &= 2 - \fxt h x - \fxt h y, \text{ \quad and} \\
\dSfxyr h \msT x y &= 2 - \fxr h x - \fxr h y,
\end{align*}
whereas $d_{\dots}(\seq x,\seq y)=0$ \ for $\seq x=\seq y$. Then, these distance functions are $(1)$-\hspace{0.25mm}centred pseudometrics on $\mZ^m$. \\

Consequently, the following relations are preorders on $\mZ^m$.
\begin{align*}
\seq x\preceq_{\dSfp f \msG} \seq y &\iff \fxp f x \le \fxp f y, \\
\seq x\preceq_{\dSft f \msG} \seq y &\iff \fxt f x \le \fxt f y, \phantom{\text{ \quad and}} \\
\seq x\preceq_{\dSfr f \msG} \seq y &\iff \fxr f x \le \fxr f y,
\end{align*}
\begin{align*}
\seq x\preceq_{\dSfp g \msQ} \seq y &\iff \fxp g x \le \fxp g y, \\
\seq x\preceq_{\dSft g \msQ} \seq y &\iff \fxt g x \le \fxt g y, \phantom{\text{ \quad and}} \\
\seq x\preceq_{\dSfr g \msQ} \seq y &\iff \fxr g x \le \fxr g y, \\
\seq x\preceq_{\dSfp h \msT} \seq y &\iff \fxp h x \ge \fxp h y, \\
\seq x\preceq_{\dSft h \msT} \seq y &\iff \fxt h x \ge \fxt h y, \text{ \quad and} \\
\seq x\preceq_{\dSfr h \msT} \seq y &\iff \fxr h x \ge \fxr h y.
\end{align*}
\end{coro}
\begin{proof}
Confer Lemma~\refs{implied_generating_functions_Z_m}, Proposition~\refs{function_metric_other}, and Corollary~\refs{induced_preorder_other}.
\end{proof}
\begin{rema}
The pseudometrics of this corollary are constructed from generating and supplementary functions as described in the previous section. In contrast to $d^{(m)}_{AM}$, $d^{(m)}_{GM}$, and $d^{(m)}_{HM}$, they do not exploit the positional information within the sequences. They simply treat sequences as multisets.
\end{rema}
\ \\[-4ex]

The centred pseudometrics from Corollary~\refs{induced_pseudometric_preorder_Z_m} and the corresponding\lb pseudometrics defined at the beginning in Lemma~\refs{pseudometric_Z_m} are in general different.\lb This may also apply to corresponding preorders. Furthermore, $d^{(m)}_{GM}$ and $d^{(m)}_{HM}$ are not $(1)$-centred pseudometrics in the sense of Lemma~\refs{function_metric}, because\lb $d^{(m)}_{GM}\bigl(\seq x,\seq y\bigr) \ne d^{(m)}_{GM}\bigl((1),\seq x\bigr) + d^{(m)}_{GM}\bigl((1),\seq y\bigr)$ \ or\\ $d^{(m)}_{HM}\bigl(\seq x,\seq y\bigr) \ne d^{(m)}_{HM}\bigl((1),\seq x\bigr) + d^{(m)}_{HM}\bigl((1),\seq y\bigr)$ \ can hold for suitable $\seq x,\seq y$.\\[-2ex]

However, for pairwise different or identical sequences, clear relationships between the corresponding pseudometries and the associated preorders can be derived. Only pseudometrics based on arithmetic means coincide. Nevertheless, with respect to the centre $(1)\in\mZ^m$ considered here, the preorders induced by related pseudometrics are identical.

We define the preorders associated with the pseudometrics according to Lemma~\refs{pseudometric_Z_m} analogously to Lemma~\refs{induced_preorder}. After proving a technical lemma for the Pythagorean means, we formulate the complete theorem.

\begin{lemm} \label{preorder_Z_m}
Let  $m\in\mN$ and $\seq x,\seq y\in\mZ^m$.

For every $1$-centred pseudometric $d$ on $\mZ$, the relations defined by
\begin{align*}
\seq x\preceq_{d^{(m)}_{AM}} \seq y &\iff d^{(m)}_{AM}\bigl((1),\seq x\bigr) \le d^{(m)}_{AM}\bigl((1),\seq y\bigr), \\
\seq x\preceq_{d^{(m)}_{GM}} \seq y &\iff d^{(m)}_{GM}\bigl((1),\seq x\bigr) \le d^{(m)}_{GM}\bigl((1),\seq y\bigr)\text{, \quad and} \\
\seq x\preceq_{d^{(m)}_{HM}} \seq y &\iff d^{(m)}_{HM}\bigl((1),\seq x\bigr) \le d^{(m)}_{HM}\bigl((1),\seq y\bigr)
\end{align*}
are preorders on $\mZ$.
\end{lemm}
\begin{proof}
The specified relation is reflexive and transitive because the ordinary \enquote{is less than or equal to} relation is reflexive and transitive on $\mR_{\ge0}$.
\end{proof}

\begin{lemm} \label{means_and_sums}
Let  $m\in\mN$ and $a_i,b_i\in\mR$. Then \\[-3ex]
\begin{align*}
(1)\hspace*{15.6mm} \frac{1}{m}\sum_{i=0}^{m-1}(a_i+b_i) &\ \ \,\,=\ \ \,\, \frac{1}{m}\sum_{i=0}^{m-1}a_i + \frac{1}{m}\sum_{i=0}^{m-1}b_i. \\
\text{If all } a_i,b_i\ge0, \text{ then also} \hspace*{30mm} \\[-1ex]
(2)\hspace*{10mm} \left(\prod_{i=0}^{m-1}(a_i+b_i)\right)^{\frac{1}{m}} &
\ \ \,\,\ge\ \ \,\, \left(\prod_{i=0}^{m-1}a_i\right)^{\frac{1}{m}}+\left(\prod_{i=0}^{m-1}b_i\right)^{\frac{1}{m}}. \hspace*{30mm} \\[0.5ex]
\text{If all } a_i,b_i>0, \text{ then additionally} \hspace*{16.2mm} \\
(3)\hspace*{22.2mm} \displaystyle{\frac{m}{\sum_{i=0}^{m-1}\frac{1}{a_i+b_i}}} &
\ \ \,\,\ge\ \ \,\,\ \displaystyle{\frac{m}{\sum_{i=0}^{m-1}\frac{1}{a_i}} + \frac{m}{\sum_{i=0}^{m-1}\frac{1}{b_i}}}.
\end{align*}
\end{lemm}
\begin{proof}
\ \\[2ex]
(1) : \quad The equation results from linearity. \\
\ \\[-1ex]
(2) : \quad We proof the inequality by forward-backward induction.

Let  $m\in\mN$ and $a_i,b_i\in\mR_{\ge0}$.
\ \\[2ex]
Base case : $m=1$, \ i.e.  $a_0+b_0\ge a_0+b_0$.
\ \\[2ex]
Forward step : $m=k \ \longrightarrow \ m=2\mdot k$. \\[-2ex]
\[ \text{We set}\qquad a=\left(\prod_{i=0}^{k-1}a_i\right)^{\frac{1}{k}}, \ b=\left(\prod_{i=0}^{k-1}b_i\right)^{\frac{1}{k}}, \ c=\left(\prod_{i=k}^{2k-1}a_i\right)^{\frac{1}{k}}, \text{ and } d=\left(\prod_{i=k}^{2k-1}b_i\right)^{\frac{1}{k}}. \\
 \]
By assumption, we know \\[-1ex]
\[ \left(\prod_{i=0}^{k-1}(a_i+b_i)\right)^{\frac{1}{k}} \ge a+b\ge0 \quad\text{and also}\quad
\left(\prod_{i=k}^{2k-1}(a_i+b_i)\right)^{\frac{1}{k}} \ge c+d\ge0. \]
Then \\[-4ex]
\begin{align*}
\left(\prod_{i=0}^{2k-1}(a_i+b_i)\right)^{\frac{1}{k}}
&= \left(\prod_{i=0}^{k-1}(a_i+b_i)\right)^{\frac{1}{k}} \mdot \left(\prod_{i=k}^{2k-1}(a_i+b_i)\right)^{\frac{1}{k}}& \\
&\ge \big(a+b\big)\mdot \big(c+d\big) \\
&= a\mdot c+a\mdot d+b\mdot c+b\mdot d \\
&= a\mdot c+2\mdot\big(a\mdot c\mdot b\mdot d\big)^{\frac{1}{2}}+b\mdot d\ +\ a\mdot d-2\mdot\big(a\mdot d\mdot b\mdot c\big)^{\frac{1}{2}}+b\mdot c \\
&= \left(\big(a\mdot c\big)^{\frac{1}{2}} + \big(b\mdot d\big)^{\frac{1}{2}}\right)^2 +\left(\big(a\mdot d\big)^{\frac{1}{2}} - \big(b\mdot c\big)^{\frac{1}{2}}\right)^2 \\
&\ge \left(\big(a\mdot c\big)^{\frac{1}{2}} + \big(b\mdot d\big)^{\frac{1}{2}}\right)^2.
\end{align*}

The assertion follows because $a\mdot c\ge0$ and $b\mdot d\ge0$. \\[-1ex]
\[ \left(\prod_{i=0}^{2k-1}\bigl(a_i+b_i\bigr)\right)^{\frac{1}{2k}} \ge (a\mdot c)^{\frac{1}{2}} + (b\mdot d)^{\frac{1}{2}} = \left(\prod_{i=0}^{2 k-1}a_i\right)^{\frac{1}{2k}}+\left(\prod_{i=0}^{2k-1}b_i\right)^{\frac{1}{2k}}. \hspace*{22.5mm} \]
\ \\[2ex]
Backward step : $m=k>1 \ \longrightarrow \ m=k-1$. \\

Let the alleged inequality hold for $m=k$, i.e. for arbitrary $a_i,b_i\in\mR{\ge0}$. Then it also holds for
\[ a_{k-1}=\left(\prod_{i=0}^{k-2}a_i\right)^{\frac{1}{k-1}} \quad\text{ and }\quad\ b_{k-1}=\left(\prod_{i=0}^{k-2}b_i\right)^{\frac{1}{k-1}}. \]
With these assumptions, we get \\[-3ex]
\begin{align*}
\left(\prod_{i=0}^{k-1}a_i\right)^{\frac{1}{k}}+\left(\prod_{i=0}^{k-1}b_i\right)^{\frac{1}{k}}  &= \left(\prod_{i=0}^{k-2}a_i \mdot \left(\prod_{i=0}^{k-2}a_i\right)^{\frac{1}{k-1}}\right)^{\frac{1}{k}} + \left(\prod_{i=0}^{k-2}b_i \mdot \left(\prod_{i=0}^{k-2}b_i\right)^{\frac{1}{k-1}}\right)^{\frac{1}{k}} \\
&= \left(\prod_{i=0}^{k-2}a_i\right)^{\frac{1}{k-1}} +\left(\prod_{i=0}^{k-2}b_i\right)^{\frac{1}{k-1}} = a_{k-1}+b_{k-1}
\end{align*}
and thus, \\[-2ex]
\begin{align*}
\left(\prod_{i=0}^{k-1}\bigl(a_i+b_i\bigr)\right)^{\frac{1}{k}} &= 
\left(\prod_{i=0}^{k-2}\bigl(a_i+b_i\bigr)\right)^{\frac{1}{k}} \mdot \bigl(a_{k-1}+b_{k-1}\bigr)^{\frac{1}{k}} \\ &\ge \left(\prod_{i=0}^{k-1}a_i\right)^{\frac{1}{k}}+\left(\prod_{i=0}^{k-1}b_i\right)^{\frac{1}{k}} = a_{k-1}+b_{k-1} . \hspace*{8.25mm}
\end{align*}
The assertion follows by rearrangement. \\[-2ex]
\begin{align*}
\left(\prod_{i=0}^{k-2}\bigl(a_i+b_i\bigr)\right)^{\frac{1}{k}} \mdot \bigl(a_{k-1}+b_{k-1}\bigr)^{\frac{1}{k}} &\ge a_{k-1}+b_{k-1} \\
\left(\prod_{i=0}^{k-2}\bigl(a_i+b_i\bigr)\right)^{\frac{1}{k}} &\ge \bigl(a_{k-1}+b_{k-1}\bigr)^{\frac{k-1}{k}} \ge0  \hspace*{11mm} \\
\left(\left(\prod_{i=0}^{k-2}\bigl(a_i+b_i\bigr)\right)^{\frac{1}{k}}\right)^{\frac{k}{k-1}} &\ge \left(\bigl(a_{k-1}+b_{k-1}\bigr)^{\frac{k-1}{k}}\right)^{\frac{k}{k-1}} \\
\left(\prod_{i=0}^{k-2}\bigl(a_i+b_i\bigr)\right)^{\frac{1}{k-1}} &\ge a_{k-1}+b_{k-1} .
\end{align*}

(3) : \quad Proof by induction.  Henceforth, we require positive $a_i,b_i$. \\
\ \\
Base case : $m=1$, \ i.e. $ \displaystyle{\frac{1}{\frac{1}{a_0+b_0}} \ge \frac{1}{\frac{1}{a_0}} + \frac{1}{\frac{1}{b_0}}} $ or equivalently $a_0+b_0\ge a_0+b_0$. \\
\ \\[1ex]
Induction step : $m=k \ \longrightarrow \ m=k+1$. \\[-3ex]
\begin{align*}
\text{Given}\qquad\frac{k}{\sum_{i=0}^{k-1}\frac{1}{a_i+b_i}} &\ge \frac{k}{\sum_{i=0}^{k-1}\frac{1}{a_i}} + \frac{k}{\sum_{i=0}^{k-1}\frac{1}{b_i}}\qquad\text{, we set}\hspace{50mm} \\
a=\frac{1}{\sum_{i=0}^{k-1}\frac{1}{a_i}},\ b&=\frac{1}{\sum_{i=0}^{k-1}\frac{1}{b_i}},\ c=a_k,\ d=b_k. \\
\text{Therefore, }\hspace{4mm} \frac{1}{\sum_{i=0}^{k-1}\frac{1}{a_i+b_i}} &\ge a+b = \frac{1}{\frac{1}{a+b}}\ , \\
\sum_{i=0}^{k-1}\frac{1}{a_i+b_i} &\le \frac{1}{a+b}\ , \\
\sum_{i=0}^{k}\frac{1}{a_i+b_i} &= \sum_{i=0}^{k-1}\frac{1}{a_i+b_i}+\frac{1}{a_k+b_k} \le \frac{1}{a+b}+\frac{1}{c+d}\ , \quad\text{and} \hspace{50mm}\\
\frac{1}{\sum_{i=0}^{k}\frac{1}{a_i+b_i}} &\ge \frac{1}{\frac{1}{a+b}+\frac{1}{c+d}}.
\end{align*}
\ \\
The assertion follows by rearranging quadratic terms.
{\small\begin{align*}
(b\mdot c-a\mdot d)^2 &\ge 0\ , \\
b^2\mdot c^2 + a^2\mdot d^2 &\ge 2\mdot a\mdot b\mdot c \mdot d\ , \\
\frac{b\mdot c}{a\mdot d}+\frac{a\mdot d}{b\mdot c} &\ge 2\ , \\
2+\frac{c}{b}+\frac{d}{b}+\frac{c}{a}+\frac{d}{a}+\frac{c}{d}+\frac{b\mdot c}{a\mdot d}&+\frac{b}{a}
+\frac{a}{b}+\frac{a\mdot d}{c\mdot b}+\frac{d}{c}+\frac{a}{d}+\frac{a}{c}+\frac{b}{d}+\frac{b}{c} \\
&\ge 4+\frac{a}{b}+\frac{a}{c}+\frac{a}{d}+\frac{b}{a}+\frac{b}{c}+\frac{b}{d}+\frac{c}{a}+\frac{c}{b}+\frac{c}{d}+\frac{d}{a}+\frac{d}{b}+\frac{d}{c}\ , \\
\left(\frac{1}{a\mdot b}+\frac{1}{a\mdot d}+\frac{1}{c\mdot b}+\frac{1}{c\mdot d}\right)
&\mdot(a\mdot c+a\mdot d+b\mdot c+b\mdot d) \\
&\ge (a+b+c+d)\mdot\left(\frac{1}{a}+\frac{1}{b}+\frac{1}{c}+\frac{1}{d}\right)\ , \\
\left(\frac{1}{a}+\frac{1}{c}\right)\mdot\left(\frac{1}{b}+\frac{1}{d}\right)
&\mdot(a+b)\mdot(c+d) \\
&\ge \left(\frac{(a+b)\mdot(c+d)}{a+b}+\frac{(a+b)\mdot(c+d)}{c+d}\right)\mdot\left(\frac{1}{a}+\frac{1}{b}+\frac{1}{c}+\frac{1}{d}\right)\ ,
\end{align*}}

{\small\begin{align*}
\left(\frac{1}{a}+\frac{1}{c}\right)\mdot\left(\frac{1}{b}+\frac{1}{d}\right) &\ge \left(\frac{1}{a+b}+\frac{1}{c+d}\right)\mdot\left(\frac{1}{a}+\frac{1}{b}+\frac{1}{c}+\frac{1}{d}\right)\ ,\hspace*{16.3mm} \\
\frac{\left(\frac{1}{a}+\frac{1}{c}\right)\mdot\left(\frac{1}{b}+\frac{1}{d}\right)}{\frac{1}{a+b}+\frac{1}{c+d}} &\ge \frac{\left(\frac{1}{a}+\frac{1}{c}\right)\mdot\left(\frac{1}{b}+\frac{1}{d}\right)}{\frac{1}{a}+\frac{1}{c}} + \frac{\left(\frac{1}{a}+\frac{1}{c}\right)\mdot\left(\frac{1}{b}+\frac{1}{d}\right)}{\frac{1}{b}+\frac{1}{d}}\ , \\
\frac{1}{\frac{1}{a+b}+\frac{1}{c+d}} &\ge \frac{1}{\frac{1}{a}+\frac{1}{c}} + \frac{1}{\frac{1}{b}+\frac{1}{d}}\ .
\end{align*}}
\!Finally, we get \\[-2ex]
{\small\begin{align*}
\frac{1}{\sum_{i=0}^{k}\frac{1}{a_i+b_i}} &\ge \frac{1}{\frac{1}{a+b}+\frac{1}{c+d}} \ge \frac{1}{\frac{1}{a}+\frac{1}{c}} + \frac{1}{\frac{1}{b}+\frac{1}{d}} =
\frac{1}{\sum_{i=0}^{k-1}\frac{1}{a_i}+\frac{1}{a_k}} + \frac{1}{\sum_{i=0}^{k-1}\frac{1}{b_i}+\frac{1}{b_k}} \hspace{3mm} \\
\text{\normalsize and thus}, \hspace{10mm}
\frac{k+1}{\sum_{i=0}^{k}\frac{1}{a_i+b_i}} &\ge \frac{k+1}{\sum_{i=0}^{k}\frac{1}{a_i}} + \frac{k+1}{\sum_{i=0}^{k}\frac{1}{b_i}} . \qedhere
\end{align*}}
\end{proof}
\ \\[2ex]

From this lemma, we deduce the relationships between the considered pseudo-\lb metrics. Although Lemma~\refs{pseudometric_Z_m} holds for arbitrary pseudometrics $d$ on $\mZ$, we recall that in this context we only further investigate $1$-centred pseudometrics on $\mZ$.

\begin{theo} \label{compressed_pseudometrics_Z_m}
Let  $m\in\mN$, $d$ be a pseudometric on $\mZ$, and $\seq x,\ \seq y\in\mZ^m$ where either\lb $x_i=y_i$ or $x_i\ne y_i$ hold for all $0\le i\le m-1$. Furthermore, let $f\in\msG$ with\lb $d(x_i,y_i)=f(x_i)+f(y_i)$ \ for $x_i\ne y_i$. Then \\[-3ex]
\begin{align*}
(1)\hspace*{10mm} d^{(m)}_{AM}\bigl(\seq x,\seq y\bigr) &\ \ \,\,=\ \ \,\,\dSfxyp f \msG x y, \\
(2)\hspace*{10.1mm} d^{(m)}_{GM}\bigl(\seq x,\seq y\bigr) &
\ \ \,\,\ge\ \ \,\,\dSfxyt f \msG x y, \text{ \quad and} \\
(3)\hspace*{9.8mm} d^{(m)}_{HM}\bigl(\seq x,\seq y\bigr) &
\ \ \,\,\ge\ \ \,\,\dSfxyr f \msG x y, \\
whereas\hspace*{29mm} \\
(4)\hspace*{11.6mm} \seq x\preceq_{d^{(m)}_{AM}} \seq y &\iff \seq x\preceq_{\dSfp f \msG} \seq y, \\
(5)\hspace*{11.7mm} \seq x\preceq_{d^{(m)}_{GM}} \seq y &\iff \seq x\preceq_{\dSft f \msG} \seq y, \text{ \quad and} \\
(6)\hspace*{11.4mm} \seq x\preceq_{d^{(m)}_{HM}} \seq y &\iff \seq x\preceq_{\dSfr f \msG} \seq y.
\end{align*}
\end{theo}

\begin{proof}
For identical $\seq x$ and $\seq y$, all assertions are trivially true. Now let $\seq x$ and $\seq y$ be pairwise distinct.

The considered pseudometrics can be expressed as terms of values of generating functions.

According to Lemma~\refs{pseudometric_Z_m}, we get
\begin{align*}
d^{(m)}_{AM}\bigl(\seq x,\seq y\bigr) &= \frac{1}{m}\sum_{i=0}^{m-1}\bigl(f(x_i)+f(y_i)\bigr), \\
d^{(m)}_{GM}\bigl(\seq x,\seq y\bigr) &= \left(\prod_{i=0}^{m-1}\bigl(f(x_i)+f(y_i)\bigr)\right)^{\frac{1}{m}}, \text{ \quad and} \\
d^{(m)}_{HM}\bigl(\seq x,\seq y\bigr) &= {\footnotesize\begin{cases}
	\displaystyle{\frac{m}{\sum_{i=0}^{m-1}\frac{1}{f(x_i)+f(y_i)}}} & \text{if } f(x_i)+f(y_i)>0 \text{ for all pairs } (x_i,y_i), \\
	0 & \text{otherwise}. \end{cases}}
\end{align*}
By Lemma~\refs{implied_generating_functions_Z_m} and Corollary~\refs{induced_pseudometric_preorder_Z_m}, we have
\begin{align*}
\dSfxyp f \msG x y &= \frac{1}{m}\sum_{i=0}^{m-1}f\seq x + \frac{1}{m}\sum_{i=0}^{m-1}f\seq y, \\
\dSfxyt f \msG x y &= \left(\prod_{i=0}^{m-1}f\seq x\right)^{\frac{1}{m}} + \left(\prod_{i=0}^{m-1}f\seq y\right)^{\frac{1}{m}}, \text{ \quad and} \\
\dSfxyr f \msG x y &= {\footnotesize\begin{cases}
	\displaystyle{\frac{m}{\sum_{i=0}^{m-1}\frac{1}{f\seq x}}} & \text{if } f\seq x>0 \text{ for all } x_i, \\
	0 & \text{otherwise} \end{cases}
+ \begin{cases}
	\displaystyle{\frac{m}{\sum_{i=0}^{m-1}\frac{1}{f\seq y}}} & \text{if } f\seq y>0 \text{ for all } y_i, \\
	0 & \text{otherwise}. \end{cases}}
\end{align*}
\ \\
(1), (2), (3) : \quad If we set $a_i=f(x_i)$ and $b_i=f(y_i)$, the assertions correspond to those of Lemma~\refs{means_and_sums}. Thus, the statements (1) and (2) are proved by the lemma because $f(x_i)\ge0$ and $f(y_i)\ge0$ by definition. 

For statement (3), we distinguish three cases. If there exist $x_j$ and $y_k$ such that $f(x_j)=f(x_k)=0$, then the alleged inequality is reduced to $\dSfxyr f \msG x y\ge0$ which is true by definition. Let w.l.o.g. be an arbitrary $f(x_i)=0$ and all $f(y_i)>0$. The remaining inequality
$\displaystyle{\frac{m}{\sum_{i=0}^{m-1}\frac{1}{f(x_i)+f(y_i)}}} \ge \displaystyle{\frac{m}{\sum_{i=0}^{m-1}\frac{1}{f\seq y}}}$ holds since $f(x_i)\ge0$ and $f(y_i)>0$. Finally, if all $f(x_i),f(y_i)>0$, then the assertions (3) of this proposition\lb and of Lemma~\refs{means_and_sums} coincide. \\
\ \\
(4), (5), (6) : \quad Summarising Lemmata~\refs{implied_generating_functions_Z_m} and \refs{preorder_Z_m}, we conclude
\begin{align*}
\seq x\preceq_{d^{(m)}_{AM}} \seq y &\iff \fxp f x \le \fxp f y, \\
\seq x\preceq_{d^{(m)}_{GM}} \seq y &\iff \fxt f x \le \fxt f y\text{, \quad and} \\
\seq x\preceq_{d^{(m)}_{HM}} \seq y &\iff \fxr f x \le \fxr f y,
\end{align*}
which corresponds exactly to the representation for the other pseudometrics in\lb Corollary~\refs{induced_pseudometric_preorder_Z_m}.
\end{proof}

Propositions~\refs{isomorphic_linear}, \refs{function_triple} and Corollary~\refs{homomorphic} have demonstrated the existence of\lb preorder-preserving isomorphisms between the set of generating functions and\lb their supplements or suitable subsets thereof in general. We apply these results to the considered pseudometrics and preorders on $\mZ^m$ according to Corollary~\refs{induced_pseudometric_preorder_Z_m} and define $\msGm_1 = \{f^{(m)}\in\msGm:f^{(m)}(\seq x)\le1\}$ and $\msT^{(m)}_0 = \{h^{(m)}\in\msT:h(\seq x)>0\}$. This results in compatible preorder-preserving isomorphisms on $\mZ^m$.

We emphasise that the same basis functions according to Definition~\refs{generating_functions_Z} induce\lb appropriate pseudometrics and preorders both on $\mZ^m$ and hence on $\mZ$ itself.\lb Moreover, there are also triples of functions $\fp f \in\msGm$, $\ft g\in\msQm$, and\lb $\ft h\in\msTmO$ with analogous properties as described in Corollary~\refs{function_triple}. \\[-2ex]

\begin{prop} \label{isomorphic_linear_Z_m}
Let $m\in\mN$ and $\seq x,\ \seq y\in\mZ^m$. Then \\[-1ex]
\[ \msGm\quad\simeq\quad\msQm \qquad\text{and}\qquad \msGm_1\quad\simeq\quad\msTm. \]
For each $f\in\msG$, there exists $g\in\msQ$, and for each $f_1\in\msG_1 = \{f\in\msG:f(x)\le1\}$, there exists $h\in\msT$ and vice versa, such that the induced preorders are pairwise equivalent on $\mZ^m$, i.e. 
\[ \seq x\preceq_{\dSfp f \msG} \seq y \iff \seq x\preceq_{\dSfp g \msQ} \seq y \quad\text{and}\quad \seq x\preceq_{\dSfp {f_1} {\msG_1}}\seq y \iff \seq x\preceq_{\dSfp h \msT} \seq y.\]
\end{prop}
\begin{proof}
Let $f,f_1\in\msG$. We set $g=f+1$ and $h=1-f_1$. Then $g\in\msQ$ and $h\in\msT$ by Lemma~\refs{generating_supplements}. The transformations are bijective and strictly monotonic. The equivalence of the corresponding preorders on $\mZ$ follows by Corollary~\refs{induced_preorder_other} and Proposition~\refs{isomorphic_linear}.

By Lemma~\refs{implied_generating_functions_Z_m}, 
$\fp f\in\msGm$, $\fp g\in\msQm$, $\fp {f_1}\in\msGm_1$, $\fp h\in\msTm$,\lb and furthermore, \\[-3ex]
\begin{align*}
\fxp f x &= \frac{1}{m}\sum_{i=0}^{m-1}f\seq x = \frac{1}{m}\sum_{i=0}^{m-1}(g\seq x-1) = 
	\fxp g x-1 \text{\quad and} \\
\fxp {f_1} x &= \frac{1}{m}\sum_{i=0}^{m-1}{f_1}\seq x = \frac{1}{m}\sum_{i=0}^{m-1}(1+h\seq x) = 
	1-\fxp h x .
\end{align*}
The transformations $\fp g=\fp f+1$ and $\fp {f_1}=1-\fp h$ are also bijective\lb and strictly monotonic. The equivalence of the preorders on $\mZ^m$ follows here by\lb Corollary~\refs{induced_pseudometric_preorder_Z_m}.
\end{proof}
\ \\[-5ex]

\begin{theo} \label{equivalent_preorder_tuple_Z_m}
Let $m\in\mN$ and $\seq x,\ \seq y\in\mZ^m$. Then \\[-1ex]
\[ \msGm\quad\simeq\quad\msQm\simeq\quad\msT^{(m)}_0. \]
There exist triples of functions $f\in\msG$, $g\in\msQ$, and $h\in\msT_0 = \{h\in\msT:h(x)>0\}$ such that the corresponding triples of functions $\fp f \in\msGm$, $\ft g\in\msQm$, and $\ft h\in\msTmO$ induce pairwise equivalent preorders on $\mZ^m$, i.e.
\[\seq x\preceq_{\dSfp f \msG} \seq y \iff \seq x\preceq_{\dSft g \msQ} \seq y \iff \seq x\preceq_{\dSft  h {\msT_0}} \seq y.\]
\end{theo}
\begin{proof}
Let $f\in\msG$. As in Proposition~\refs{function_triple}, we set $g=\e^{\ f}\in\msQ$, $h=\e^{-f}\in\msT_0$, and thus, $h=\frac 1 g$. The transformations are bijective and strictly monotonic. The equivalence of the corresponding preorders on $\mZ$ follows by Corollary~\refs{induced_preorder_other} and Proposition~\refs{function_triple}. \\[-2ex]

According to the assumptions and Lemma~\refs{implied_generating_functions_Z_m}, we get $\fp f\in\msGm$, $\ft g\in\msQm$, and $\ft h\in\msTmO$. Then
\begin{align*}
\fxp f x &= \frac{1}{m}\sum_{i=0}^{m-1}f\seq x = \ln\left(\fxt g x\right) = -\ln\left(\fxt h x\right), \\
\fxt g x &= \left(\prod_{i=0}^{m-1}g\seq x\right)^{\frac{1}{m}} = \left(\prod_{i=0}^{m-1}\e^{\ f(x_i)}\right)^{\frac{1}{m}} = \e^{\,\frac{1}{m}\sum_{i=0}^{m-1}f\seq x} = \e^{\ \fxp f x}, \text{\quad and} \\
\fxt h x &= \left(\prod_{i=0}^{m-1}h\seq x\right)^{\frac{1}{m}} = \left(\prod_{i=0}^{m-1}\e^{-f(x_i)}\right)^{\frac{1}{m}} = \e^{-\frac{1}{m}\sum_{i=0}^{m-1}f\seq x} = \e^{-\fxp f x}.
\end{align*}

The corresponding transformations $\ft g=\e^{\ \fp f}$, $\ft h=\e^{-\fp f}$, and\lb $\ft h=\frac 1 {\ft g}$ are also bijective and strictly monotonic. The equivalence of the\lb preorders on $\mZ^m$ follows here by Corollary~\refs{induced_pseudometric_preorder_Z_m}.
\end{proof}
\ \\[-3ex]

As a result of the use of different Pythagorean means, there are two additional pairs of functions with similar properties. The related preorders of the function pairs $\fp g\in\msQm$, $\fr h\in\msTmO$ and $\fp h\in\msTmO$, $\fr g\in\msQm$ are also equivalent if $g$ and $h$ are reciprocals of each other. By Proposition~\refs{function_triple} and Theorem~\refs{equivalent_preorder_tuple_Z_m}, in this case $g$ and $h$ always belong to a basic triple, which is completed by the function $f=\ln(g)=-\ln(h)\in\msG$.

However, the remaining two generating functions $\ft f,\fr f\in\msGm$ defined in Lemma~\refs{implied_generating_functions_Z_m} cannot generally be represented as mean values of other basis functions. \\[-1ex]

\begin{theo} \label{equivalent_preorder_pairs_Z_m}
Let $m\in\mN$, $\seq x,\ \seq y\in\mZ^m$, $g\in\msQ$, and $h\in\msT_0$ with $h=\frac{1}{g}$. Then
\[ \msQm\simeq\quad\msT^{(m)}_0 \]
such that
\[ \seq x\preceq_{\dSfp g \msQ} \seq y \iff \seq x\preceq_{\dSfr h {\msT_0}} \seq y \quad\text{and}\quad \seq x\preceq_{\dSfp h {\msT_0}} \seq y \iff \seq x\preceq_{\dSfr g \msQ} \seq y.\]
\end{theo}
\pagebreak
\begin{proof}
The transformation $h=\frac 1 g$ is bijective and strictly monotonic. The equivalence of the corresponding preorder on $\mZ$ corresponds to Proposition~\refs{function_triple} and Theorem~\refs{equivalent_preorder_tuple_Z_m}.

According to Lemma~\refs{implied_generating_functions_Z_m}, we know $\fp g,\fr g\in\msQm$ and $\fp h,\fr h\in\msTmO$.
By assumption, we get
\begin{align*}
\fxp g x &= \frac{1}{m}\sum_{i=0}^{m-1}g\seq x = \frac{1}{m}\sum_{i=0}^{m-1}\frac{1}{h\seq x} = \frac{1}{\fxr h x}, \\
\fxp h x &= \frac{1}{m}\sum_{i=0}^{m-1}h\seq x = \frac{1}{m}\sum_{i=0}^{m-1}\frac{1}{g\seq x} = \frac{1}{\fxr g x}.
\end{align*}

The corresponding transformations $\fr h=\frac{1}{\fp g}$ and $\fr g=\frac{1}{\fp h}$ are bijective and strictly monotonic. The assertion again follows from Corollary~\refs{induced_pseudometric_preorder_Z_m}.
\end{proof}
\ \\[-3ex]

Consequently, there exist a preorder-preserving homomorphism on $\msQm$ and\lb another isomorphism between $\msQm$ and $\msTm_0$  which complete the corresponding\lb isomorphisms of Proposition~\refs{isomorphic_linear_Z_m} and the Theorems~\refs{equivalent_preorder_tuple_Z_m} and \refs{equivalent_preorder_pairs_Z_m}. \\[-1ex]

\begin{coro} \label{homomorphic_Z_m} 
Let $m\in\mN$, $\seq x,\ \seq y\in\mZ^m$, and $g\in\msQ$. There exists a homomorphism $\digamma:\msQm\to\msQm$ such that
\[ \seq x\preceq_{\dSfp g \msQ} \seq y  \iff \seq x\preceq_{\dSft {\digamma(g)} {\msQ}} \seq y.\]

Furthermore, again $\msQm\simeq\msTm_0$. For each $g\in\msQm$, there exists another $h\in\msTm_0$\lb and vice versa, such that the induced preorders are pairwise equivalent on $\mZ^m$, i.e.
\[ \seq x\preceq_{\dSfp g \msQ} \seq y \iff \seq x\preceq_{\dSft h {\msT_0}} \seq y.\]
\end{coro}
\begin{proof}
The morphisms $\digamma(g)=\e^{\ g-1}$ and $h=\e^{\ 1-g}$, i.e.\\ $g=\ln(\digamma(g))+1=1-\ln(h)$ and $h=\frac{1}{\digamma(g)}$, meet the requirements. They are\lb bijective and strictly monotonic. The equivalence of the corresponding preorders\lb on $\mZ$ corresponds to Corollary~\refs{homomorphic}. So, $g-1\in\msG$, $\digamma(g)\in\msQ$, and $\frac{1}{\digamma(g)}\in\msT_0$. According to Lemma~\refs{implied_generating_functions_Z_m}, we have $\fp {(g-1)},\ \fp {\ln(\digamma(g))}\in\msGm$, $\fp g,\ \ft {\digamma(g)}\in\msQm$, and $\ft {\frac{1}{\digamma(g)}}\in{\msTm_0}$.

The asserted equivalence of the preorders on $\mZ^m$ follows from Proposition~\refs{isomorphic_linear_Z_m} and Theorem~\refs{equivalent_preorder_tuple_Z_m}. Based on them,
\begin{align*}
\seq x\preceq_{\dSfp {g-1} \msG} \seq y &\iff \seq x\preceq_{\dSfp g \msQ} \seq y \qquad\text{and} \hspace{66mm} \\
\seq x\preceq_{\dSfp {\ln(\digamma(g))} \msG} \seq y &\hspace{23mm} \iff \seq x\preceq_{\dSft {\digamma(g)} \msQ} \seq y \iff \seq x\preceq_{\dSft {\frac{1}{\digamma(g)}} {\msT_0}} \seq y.
\end{align*}
\ \\[-8ex]
\end{proof}
\pagebreak

The finishing diagram of this section depicts isomorphisms between the set of\lb generating functions and its supplements and related subsets described in this\lb section. It extends Figure~\refs{dia1} from the introduction for $\mZ^m$.

\begin{figure}[H]
  \centering\vspace*{-3mm}
\begin{gather*}\xymatrixcolsep{2pc}\xymatrixrowsep{4pc}\xymatrix{
	\msGm_1
		\ar@{..}[rrrrr]|-{\ \ \subseteq\ \ }
		\ar@{<->}[dd]_-{\ref{isomorphic_linear_Z_m}}
	    & 
		    & & & & \msGm
				\ar@{<->}[d]_-{\ref{equivalent_preorder_tuple_Z_m}}
				\ar@<10pt>@{<->}@/^1.5pc/[dd]^-{\ref{equivalent_preorder_tuple_Z_m}} \\
		& \msQ_{\phantom{j}}^{(m)}
			\ar@{<->}[rrrru]^-(.564){\ref{isomorphic_linear_Z_m}\,}
			\ar@{<->}[rrrr]_-(.52){\ref{homomorphic_Z_m}}
			\ar@{<->}[rrrrd]_-(.585){\ref{homomorphic_Z_m}\,\,}\
		    & & & & \msQm
				\ar@{<->}[d]_-{\ref{equivalent_preorder_tuple_Z_m}}
				                     ^-{\ref{equivalent_preorder_pairs_Z_m}} \\
	\msTm
		\ar@{..}[rrrrr]|-{\ \ \supseteq\ \ }
		& 
		    & & & & \msTm_0
}\end{gather*}
   \caption{Preorder-preserving isomorphisms between the sets of generating and\\ supplementary functions on $\mZ^m$.} \label{dia3}
\end{figure}
\

This section started with the definition of the pseudometrics $d^{(m)}_{AM}$, $d^{(m)}_{GM}$, and $d^{(m)}_{HM}$ according to Lemma~\refs{pseudometric_Z_m}. From them, we derived $(1)$-centred pseudometrics in the sense of Lemma~\refs{function_metric}, whose definitions were summarised in Corollary~\refs{induced_pseudometric_preorder_Z_m}. Based on the latter, we provided a general and comprehensive presentation of the possibilities of extending the ideas presented in the previous section for creating pseudometrics and preorders to sets of integer sequences.

Although the considered $(1)$-centred pseudometrics do not exploit the positional information within the sequences, as is the case with $d^{(m)}_{AM}$, $d^{(m)}_{GM}$ and $d^{(m)}_{HM}$, there are definite relationships between the corresponding pseudometrics and the associated preorders under the condition of pairwise different or identical sequences. Related pseudometrics induce identical preorders. Theorem~\refs{compressed_pseudometrics_Z_m} expresses these key results.

\pagebreak


\section{Sequences of consecutive integers} \label{SCI}

The main focus of this paper is on sets of finite integer sequences, in particular on sequences of consecutive integers. So far, we explored the set $\mZ^m$ for a given $m\in\mN$, i.e. the set of $m$-tuples of integral numbers which can be equivalently interpreted as sequences of integers of length $m$. Such $m$-tuples can contain equal values.
 
However, sets of $m$ integers only include different values. Ordered, they may be regarded as strictly increasing sequences of $m$ integers. So, they form a  subset of $\mZ^m$, namely $\{ \seq x \in\mZ^m : x_i<x_{i+1},\ i=0,\dots,m-2 \}$ if $m>1$.

A special case of a strictly increasing sequence of integers is a sequence of\lb consecutive integers. Here $x_{i+1}=x_i+1$ or equivalently $x_i = x_0+i$. Thus, we get a subset of the above subset of $\mZ^m$. Our next step is to examine this set of sequences of consecutive integers of length $m$ in more detail.

\begin{defi} \label{consecutive_x_m}
Let  $m\in\mN$ and $x,x_i\in\mZ$.\\
The finite sequence of $m$ consecutive integers $\seq x _{i=0}^{m-1} = (x+i)_{i=0}^{m-1}$, where $x=x_0$, is denoted by $\ap x m$ \cite{Ziller_2020}. The first member of $\ap x m$ is $x$. In cases where the length\lb is obvious, we write $\as x$ for short.
\end{defi}
\begin{defi} \label{consecutive_Z_m}
Let  $m\in\mN$.\\
The set of all sequences of consecutive integers $\as x$ of length $m$ is denoted by
\ \\[-2ex]
\[\ap \mZ m=\{\as x : x\in\mZ\}.\]
\end{defi}

The sequence $\as x$ is completely determined by $x$. By definition, $\ap \mZ m \subset \mZ^m$.\lb Moreover, there exists a bijection $\as x \mapsto x$ between $\ap \mZ m$ and $\mZ$. This raises the question of how the associated pseudometrics and preorders defined so far should be classified in this context.
\ \\\\[-3ex]

In the past section, we started with the pseudometrics defined in Lemma~\refs{pseudometric_Z_m}, which exploit the positional information within the sequences. In this context, they would be
\ \\[-3ex]
\begin{align*}
d^{(m)}_{AM}\bigl(\as x,\as y\bigr) &= \frac{1}{m}\sum_{i=0}^{m-1}d(x+i,y+i), \\
d^{(m)}_{GM}\bigl(\as x,\as y\bigr) &= \left(\prod_{i=0}^{m-1}d(x+i,y+i)\right)^{\frac{1}{m}}\text{, \quad and} \\
d^{(m)}_{HM}\bigl(\as x,\as y\bigr) &= \begin{cases}
	\displaystyle{\frac{m}{\sum_{i=0}^{m-1}\frac{1}{d(x+i,y+i)}}} & \text{if } d(x+i,y+i)>0 \text{ for all pairs } (\as x,\as y), \\
	0 & \text{otherwise.} \end{cases}
\end{align*}

We restrict ourselves to the associated $(1)$-\hspace{0.25mm}centred pseudometrics
according\lb to Corollary~\refs{induced_pseudometric_preorder_Z_m} in connection with sequences of consecutive integers. A sequence of\lb consecutive integers can also be regarded as an interval of integers, i.e. a set. Then\lb positional information is no longer relevant. The elements of such a sequence or\lb interval are distinct. Thus, the requirements of Theorem~\refs{compressed_pseudometrics_Z_m} are satisfied in a\lb non-trivial manner. The associated pseudometrics each lead to identical preorders.

We would like to point out that $(1)$ does not occur in the set of sequences of\lb consecutive integers, so $(1)\not\in\ap \mZ m$ but $(1)\in\mZ^m$. However, restricting the\lb pseudometrics to $\ap \mZ m$ preserves the property of being $(1)$-\hspace{0,25mm}centred.
\ \\

Let $m\in\mN$, $f:\mZ\to\mR_{\ge0}$. By Lemma~\refs{implied_generating_functions_Z_m}, the Pythagorean means of $f$ along\lb a sequence of consecutive integers $\as x\in\ap \mZ m$ result in
\ \\[-2ex]
\begin{align*}
\fxcp f x &= \frac{1}{m}\sum_{i=0}^{m-1}f(x+i), \\
\fxct f x &= \left(\prod_{i=0}^{m-1}f(x+i)\right)^{\frac{1}{m}} \text{, \quad and} \\
\fxcr f x &= \begin{cases}
	\displaystyle{\frac{m}{\sum_{i=0}^{m-1}\frac{1}{f(x+i)}}} & \text{if } f(x+i)>0 \text{ for all } i, \\
	0 & \text{otherwise.} \end{cases}
\end{align*}

These formulae correspond to moving averages. The moving arithmetic average \cite{Box_2015} and the moving geometric average \cite{Box_2015, Pohlen_1929} have been extensively studied in the field of  time series analysis. Here, we also use the moving harmonic mean analogously.

In financial applications, the simple moving average is often understood as the\lb unweighted arithmetic mean of the previous m data points. This leads to shifted changes in the means and the data. To avoid such effects, the central mean is\lb sometimes used, which is taken from an equal number of data on either side.

However, in our applications we use the means of the function values of prospective sequence elements. So in this case, the lags have the opposite sign than for the simple moving averages.

\begin{defi} \label{moving_averages_f}
Let $m\in\mN$, $x\in\mZ$, $f:\mZ\to\mR_{\ge0}$. \\
We denote the $m^{th}$ order moving Pythagorean averages of $f$ by

\begin{align*}
\fmap f(x) &= \frac{1}{m}\sum_{i=0}^{m-1}f(x+i) \hspace{45.7mm} = \fxcp f x, \\
\fmat f(x) &= \left(\prod_{i=0}^{m-1}f(x+i)\right)^{\frac{1}{m}} \hspace{40.7mm} = \fxct f x \text{, \quad and} \\
\fmar f(x) &= \begin{cases}
	\displaystyle{\frac{m}{\sum_{i=0}^{m-1}\frac{1}{f(x+i)}}} & \text{if } f(x+i)>0 \text{ for all } i, \\
	0 & \text{otherwise} \end{cases}
	= \fxcr f x.
\end{align*}
\end{defi}

\begin{rema}
This definition clarifies the existence of a bijection between $\ap \mZ m$ and $\mZ$, such that all three Pythagorean means on $\ap \mZ m$ correspond to the respective moving averages on $\mZ$. The function $\nu\bigl(\as x\bigr)=x$ and its inverse $\nu^{-1}(x)=\as x$ are both unique.
\end{rema}

As follows from Corollary~\refs{induced_pseudometric_preorder_Z_m}, the adaptation of pseudometrics and preorders\lb on $\mZ^m$ to sequences of consecutive integers now also leads to the use of moving\lb averages. Let $\as x,\as y\in \ap \mZ m$,  $x,y\in\mZ$, $f\in\msG$, $g\in\msQ$, and $h\in\msT$. For $x\ne y$,\lb we get
\ \\[-3ex]
\begin{align*}
\dSfxycp f \msG x y &= \fmap f(x) + \fmap f(y), \\
\dSfxyct f \msG x y &= \fmat f(x) + \fmat f(y), \\
\dSfxycr f \msG x y &= \fmar f(x) + \fmar f(y), \\
\dSfxycp g \msQ x y &= \fmap g(x) + \fmap g(y) - 2, \\
\dSfxyct g \msQ x y &= \fmat g(x) + \fmat g(y)- 2, \\
\dSfxycr g \msQ x y &= \fmar g(x) + \fmar g(y)- 2, \\
\dSfxycp h \msT x y &= 2 - \fmap h(x) - \fmap h(y)),\\
\dSfxyct h \msT x y &= 2 - \fmat h(x) - \fmat h(y), \text{ \quad and} \\
\dSfxycr h \msT x y &= 2 - \fmar h(x) - \fmar h(y),
\end{align*}
whereas $d_{\dots}(x,y)=0$ \ for $x=y$.\\

Consequently, the following relations are preorders on $\ap \mZ m$.
\begin{align*}
\as x\preceq_{\dSfp f \msG} \as y &\iff \fmap f(x) \le \fmap f(y), \\
\as x\preceq_{\dSft f \msG} \as y &\iff \fmat f(x) \le \fmat f(y), \\
\as x\preceq_{\dSfr f \msG} \as y &\iff \fmar f(x) \le \fmar f(y), \\
\as x\preceq_{\dSfp g \msQ} \as y &\iff \fmap g(x) \le \fmap g(y), \\
\as x\preceq_{\dSft g \msQ} \as y &\iff \fmat g(x) \le \fmat g(y), \\
\as x\preceq_{\dSfr g \msQ} \as y &\iff \fmar g(x) \le \fmar g(y), \\
\as x\preceq_{\dSfp h \msT} \as y &\iff \fmap h(x) \ge \fmap h(y), \\
\as x\preceq_{\dSft h \msT} \as y &\iff \fmat h(x) \ge \fmat h(y), \text{ \quad and} \\
\as x\preceq_{\dSfr h \msT} \as y &\iff \fmar h(x) \ge \fmar h(y).
\end{align*}

The relation between values of basis functions obviously depends on the arguments. Therefore, in general, e.g. $f(x) \le f(y) \quad\centernot\Longrightarrow\quad f(x+1) \le f(y+1)$, and as well\lb
$f(x) \le f(y) \quad\centernot\Longrightarrow\quad \fmap f(x) \le \fmap f(y)$ \ and anything similar.
\ \\[-2ex]

Also other moving averages do not follow the relation of the first elements of the sequences. The corresponding preorder on $\mZ$ induced by a given function may differ from that on $\ap \mZ m$ induced by a suitable moving average if $m>1$.

\begin{exam}
To illustrate this, we use the number of distinct prime divisors $\omega(x)$. It is an additive function, $\omega(x)\in\msA$.
\ \\[-2ex]

$f(x)=\omega(x) = |\{p\in\mP : p\mdiv x\}|= \sum_{p | x} 1$.\\[0.5ex]
Let now $x=13$, $y=6$, $n=17$, and $m=3$.  Then
\ \\[-2ex]

$\fn f(x)\hspace{7.3mm}=f(13)=1 \qquad \le \qquad \fn f(y)\hspace{7.2mm}=f(6)=2$ \qquad and

$\fn f(x+1)=f(14)=2 \qquad \ge \hspace{-0.1mm}\qquad \fn f(y+1)=f(7)=1$.\\[0.5ex]
With
\ \\[-2ex]

$\fn f(x+2)=f(15)=2 \qquad \text{\!and\!} \hspace{0.1mm}\qquad \fn f(y+2)=f(8)=1$,\\[0.5ex]
we get
\ \\[-2ex]

$\fnp f \bigl(\as x\bigr)\hspace{10.6mm}=\frac{5}{3} \hspace{-0.3mm}\qquad \ge \qquad \fnp f \bigl(\as y\bigr)\hspace{8.4mm}=\frac{4}{3}$,

$\fnt f \bigl(\as x\bigr)\hspace{11.2mm}=\sqrt[\leftroot{4}3]{4} \hspace{0.4mm}\quad \ge \qquad  \fnt f \bigl(\as y\bigr)\hspace{9.0mm}=\sqrt[\leftroot{4}3]{2}$, \qquad and

$\fnr f \bigl(\as x\bigr)\hspace{10.7mm}=\frac{3}{2} \hspace{-0.3mm}\qquad \ge \qquad  \fnr f \bigl(\as y\bigr)\hspace{8.5mm}=\frac{6}{5}$. \\[-2ex]
\end{exam}

However, the projection of a preorder on $\ap \mZ m$ onto $\mZ$ according to the bijection $\as x \mapsto x$ also leads to a preorder on $\mZ$, which in general is not induced by a generating function or its supplements according to Definition~\refs{generating_functions_Z}.
\ \\

There are some common properties of basis functions and their moving averages that we would like to analyse. Moving averages of non-negative functions are\lb generally also non-negative functions. If even $f\in\msG$, $g\in\msQ$, and $h\in\msT$, then\lb the moving averages have the same corresponding co-domains.

\begin{lemm} \label{co-domains}
Let $m\in\mN$, $x\in\mZ$, and $f:\mZ\to\mR_ { \ge0}$.\\[0.25ex] Then also $\fmap f, \fmat f, \fmar f:\mZ\to\mR_ { \ge0}$.
\ \\[-1ex]

If furthermore, $f\in\msG$, $g\in\msQ$, and $h\in\msT$, then \\[-3ex]
\begin{align*}
\fmap f(x)&\ge0,\hspace{11mm} \fmat f(x)\ge0,\hspace{11mm} \fmar f(x)\ge0, \\
\fmap g(x)&\ge1,\hspace{11.1mm} \fmat g(x)\ge1,\hspace{11.1mm} \fmar g(x)\ge1, \text{ \quad and \quad }\\
0\le \fmap h(x)&\le1,\hspace{3.2mm} 0\le \fmat h(x)\le1,\hspace{3.2mm} 0\le \fmar h(x)\le1.
\end{align*}
\end{lemm}
\begin{proof}
For positive $x_i\in\mR$, we know the general inequality for the Pythagorean means
\ \\[-1.5ex]
\[ \min(x_i) \le HM(x_i) \le GM(x_i) \le AM(x_i) \le \max(x_i). \]
\ \\[-2ex]
The assertions follow by definition.
\end{proof}

However, for $m>1$, only the moving geometric average and the moving harmonic average of a function $f\in\msG$ are also generating functions on $\mZ$ because
\ \\[-1.5ex]
\[  f(1)=0\quad\Longrightarrow\quad \fmat f(1)=0 \quad\land\quad \fmar f(1)=0. \]
\ \\[-2.5ex]
If furthermore $f(x)\ge0$ for $x\in\mZ\ne1$ then also $\fmat f(x)\ge0$ as well as $\fmar f(x)\ge0$.
\ \\[-2ex]

Under very special conditions, the reverse even applies. One must take into\lb account that $f(1)=0$ also has the consequence that $\fmat f(x)=0$ and $\fmar f(x)=0$\lb for all $-m+2\le x\le 0$.
\ \\[-2ex]

The other moving averages considered can violate the requirements for $x=1$.\\
If e.g. $f(2)>0$, $g(2)>1$, and $h(2)<1$, then $\fmap f(1)>0$ as well as\\ 
$\fmap g(1)>1$, $\fmat g(1)>1$,  $\fmar g(1)>1$, $\fmap h(1)<1$, $\fmat h(1)<1$, and $\fmar h(1)<1$.
\ \\\\[-2ex]

Nevertheless, even taking into account the conditions $\fmat f(x)=0$ and\lb $\fmar f(x)=0$ for all $-m+2\le x\le 0$, a reverse function exists in $\msG$ only in the case of the moving geometric average. The moving harmonic average also has an analogous reverse function. But it can contain negative values.
\ \\[-2ex]

\begin{exam}
Let $f\in\msG$ and $m=2$. We assume there is a function $r\in\msG$ such that $f=\fmar r$. Then \\[-3ex]
\begin{align*}
f(x) &= \frac{x}{\frac{1}{r(x)}+\frac{1}{r(x+1)}}\ , \quad\text{ i.e. }\quad \frac{2}{f(x)} = \frac{1}{r(x)}+\frac{1}{r(x+1)}\ . \quad\text{ Thus } \\
\frac{2}{f(2)} &= \frac{1}{r(2)}+\frac{1}{r(3)}\hspace*{2.7mm}, \quad\phantom{\text{ i.e. }}\quad 
\frac{1}{r(3)}=\frac{2}{f(2)}-\frac{1}{r(2)}\ , \\
\frac{2}{f(3)} &= \frac{1}{r(3)}+\frac{1}{r(4)}\hspace*{2.7mm}, \quad\phantom{\text{ i.e. }}\quad 
\frac{1}{r(4)}=\frac{2}{f(3)}-\frac{1}{r(3)}\ , \\
&\phantom{    = \frac{1}{r(3)}+\frac{1}{r(4)}\hspace*{3.9mm}, \quad\text{ i.e. }}\quad 
\frac{1}{r(4)}=\frac{2}{f(3)}-\frac{2}{f(2)}+\frac{1}{r(2)}\ , \\
\frac{2}{f(4)} &= \frac{1}{r(4)}+\frac{1}{r(5)}\hspace*{2.7mm}, \quad\phantom{\text{ i.e. }}\quad 
\frac{1}{r(5)}=\frac{2}{f(4)}-\frac{1}{r(4)}\ , \\
&\phantom{    = \frac{1}{r(3)}+\frac{1}{r(4)}\hspace*{3.9mm}, \quad\text{ i.e. }}\quad 
\frac{1}{r(5)}=\frac{2}{f(4)}-\frac{2}{f(3)}+\frac{2}{f(2)}-\frac{1}{r(2)} \ge0\ , \quad\text{ and } \\
\frac{1}{r(2)}&\le\frac{2}{f(4)}-\frac{2}{f(3)}+\frac{2}{f(2)}\ .
\end{align*}
For $f(2)=4$, $f(2)=1$, $f(4)=4$, we obtain the contradiction $-1\le r(2)<0$.
\end{exam}

\pagebreak

\begin{prop} \label{reverse_MGHA}
Let $m\in\mN$, $x\in\mZ$, and $f\in\msG$ with $f(x)=0$ \ for $-m+2\le x\le1$\lb and $f(x)>0$ otherwise. Then, there exists a function $r\in\msG$ such that $f=\fmat r$.
\end{prop}
\ \\[-5ex]
\begin{proof}
For $m=1$, the assertion simply holds if $r=f=r^{(1,\ldot)}_{MA}$. Let now $m\in\mN>1$\lb and $f$ as assumed.
\ \\[-2ex]

We each choose $m-1$ arbitrary positive real values $y_i\in\mR>0$ with $2\le i\le m$,\lb and $z_i\in\mR>0$ with $-m+2\le i\le 0$. The function $r$ defined recursively by \\[-2ex]
\begin{align*}
r(1) &= 0, \\
r(i) &= y_i, \hspace{61.1mm} 2\le i\le m, \\
r(m+1) &= \frac{\big(f(2)\big)^m}{\prod_{i=2}^m y_i}, \\
r(m+i) &= \left(\frac{f(i+1)}{f(i)}\right)^m \mdot r(i), \hspace{40mm} i>1, \\
r(i) &= z_i, \hspace{49.3mm} -m+2\le i\le 0, \\
r(-m+1) &= \frac{\big(f(-m+1)\big)^m}{\prod_{i=-m+2}^0 z_i}, \\
r(-m+i) &= \left(\frac{f(-m+i)}{f(-m+i+1)}\right)^m \mdot r(i), \hspace{27.9mm} i<1,
\end{align*}
\ \\[-1ex]
meet the requirements.
\ \\[-2ex]

Then $r(x)>0$ \ for $x\ne1$. The corresponding moving average is \\[-2ex]
\begin{align*}
-m+2\le x\le1\ :& \hspace{118mm}\\
\fmat r(x)&=\left(\prod_{i=x}^{x+m-1} r(i)\right)^{\frac{1}{m}}
=\left(\prod_{i=x}^0 r(i) \cdot r(1) \cdot \prod_{i=2}^{x+m-1} r(i)\right)^{\frac{1}{m}}=0=f(x), \\
x=2\phantom{-m+2\le\ }:& \\
\fmat r(2)&=\left(\prod_{i=0}^{m-1} r(2+i)\right)^{\frac{1}{m}}
=\left(\prod_{i=2}^m r(i) \mdot r(m+1)\right)^{\frac{1}{m}} \\
&=\left(\prod_{i=2}^m y_i \mdot  \frac{\big(f(2)\big)^m}{\prod_{i=2}^m y_i}\right)^{\frac{1}{m}}=f(2),
\end{align*}

\begin{align*}
x>2\phantom{-m+2\le\ }:& \\
\fmat r(x)&=\left(\prod_{i=0}^{m-1} r(x+i)\right)^{\frac{1}{m}}=\left(\prod_{i=x}^{m+x-2} r(i) \mdot r(m+x-1)\right)^{\frac{1}{m}} \hspace*{20.2mm}\\
&=\left(\prod_{i=x}^{m+x-2} r(i) \mdot \left(\frac{f(x)}{f(x-1)}\right)^m \mdot r(x-1)\right)^{\frac{1}{m}} \\
&=\left(\prod_{i=x-1}^{m+x-2} r(i) \mdot \left(\frac{f(x)}{f(x-1)}\right)^m\right)^{\frac{1}{m}} \\
&=\left(\big(f(x-1)\big)^m \mdot \left(\frac{f(x)}{f(x-1)}\right)^m\right)^{\frac{1}{m}}
=f(x), \\
x=-m+1\phantom{\le1\ }:& \\
\fmat r(-m+1)&=\left(r(-m+1) \mdot \prod_{i=-m+2}^0 r(i)\right)^{\frac{1}{m}} \\
&=\left(\frac{\big(f(-m+1)\big)^m}{\prod_{i=-m+2}^0 z_i} \mdot \prod_{i=-m+2}^0 z_i\right)^{\frac{1}{m}}=f(-m+1), \\
x\le-m\phantom{m+2\le}:& \\
\fmat r(x)&=\left(r(x) \mdot \prod_{i=x+1}^{x+m-1} r(i)\right)^{\frac{1}{m}} \\
&=\left(\left(\frac{f(x)}{f(x+1)}\right)^m \mdot r(m+x) \mdot \prod_{i=x+1}^{x+m-1} r(i)\right)^{\frac{1}{m}} \\
&=\left(\left(\frac{f(x)}{f(x+1)}\right)^m \mdot \big(f(x+1)\big)^m\right)^{\frac{1}{m}}
=f(x). \qedhere
\end{align*}
\end{proof}

\pagebreak

We now turn back to arithmetic functions. All previous results apply to them\lb accordingly. For every non-negative arithmetic function $f:\mN\to\mR_{\ge0}$, we have\lb defined a function $\fn f:\mZ\to\mR_{\ge0}$ according to Lemma~\refs{AF_extend} as the basis for\lb generalising pseudometrics and preorders to $\mZ$. We want to find out relationships between corresponding pseudometrics and preorders on $\mZ$ and $\ap \mZ m$. We would like to take a closer look at the admissible arithmetic functions defined in Definition~\refs{admissible}\lb in particular. \\

Let $f\in\msA\cup\msM\cup\msI$ be an admissible arithmetic function and $n\in\mN>1$. Lemma~\refs{AF_extend} defines $\fn f$ and proves that it is one of the basis functions mentioned above. Based on Definition~\refs{moving_averages_f}, we write the Pythagorean means of $\fn f$ along a sequence of consecutive integers $\as x\in\ap \mZ m$ as
\begin{align*}
\fnxcp f x &= \frac{1}{m}\sum_{i=0}^{m-1}\fn f(x+i) \hspace{49.1mm} = \fnmap f(x), \\
\fnxct f x &= \left(\prod_{i=0}^{m-1}\fn f(x+i)\right)^{\frac{1}{m}} \hspace{44.1mm} = \fnmat f(x) \text{, \quad and} \\
\fnxcr f x &= \begin{cases}
	\displaystyle{\frac{m}{\sum_{i=0}^{m-1}\frac{1}{\fn f(x+i)}}} & \text{if } \fn f(x+i)>0 \text{ for all } i, \\
	0 & \text{otherwise} \end{cases}
	= \fnmar f(x).
\end{align*}
This notation leads to the following conclusion.
\ \\[-1ex]
\begin{coro} \label{MA_fn}
Let $x\in\mZ$, $m\in\mN$, $n\in\mN>1$, and $f:\mN\to\mR_{\ge0}$ be a non-negative arithmetic function. Then
\begin{align*}
\fnmap f(x)&= \fn {\left(\fmap f\right)}(x), \\
\fnmat f (x) &= \fn {\left(\fmat f\right)}(x) \text{, \quad and} \\
\fnmar f (x) &= \fn {\left(\fmar f\right)}(x).
\end{align*}
\end{coro}

\begin{proof}
The corollary is a direct consequence of Lemma~\refs{AF_extend} and Definition~\refs{moving_averages_f}. We get
\ \\[-3ex]
\begin{align*}
\fnmap f(x) &= \frac{1}{m}\sum_{i=0}^{m-1}\fn f(x+i) = \frac{1}{m}\sum_{i=0}^{m-1}f\bigl(1+ (x+i-1) \modo n\bigr) \hspace*{31.2mm}\\
&= \fmap f\bigl(1+ (x-1) \modo n\bigr) = \fn {\left(\fmap f\right)}(x),
\end{align*}

\begin{align*}
\fnmat f(x) &= \left(\prod_{i=0}^{m-1}\fn f(x+i)\right)^{\frac{1}{m}} = \left(\prod_{i=0}^{m-1}f\bigl(1+ (x+i-1) \modo n\bigr)\right)^{\frac{1}{m}} \\
&= \fmat f\bigl(1+ (x-1) \modo n\bigr) = \fn {\left(\fmat f\right)}(x) \text{, \qquad\qquad\qquad\qquad and} \\
\fnmar f(x) &= \begin{cases}
	\displaystyle{\frac{m}{\sum_{i=0}^{m-1}\frac{1}{\fn f(x+i)}}} & \text{if } \fn f(x+i)>0 \text{ for all } i, \\
	0 & \text{otherwise} \end{cases} \\
	&= \begin{cases}
	\displaystyle{\frac{m}{\sum_{i=0}^{m-1}\frac{1}{f\bigl(1+ (x+i-1) \modo n\bigr)}}} & \text{if } f\bigl(1+ (x+i-1) \modo n\bigr)>0 \text{ for all } i, \\
	0 & \text{otherwise} \end{cases} \\
&= \fmar f\bigl(1+ (x-1) \modo n\bigr) = \fn {\left(\fmar f\right)}(x). \qedhere
\end{align*}
\end{proof}
\ \\[-1ex]

In Lemma~\refs{periodic_fn}, we proved that the functions $\fn f$ are periodic with period $n$.\lb Consequently, this also applies to their moving averages.
\ \\[-1ex]
\begin{lemm} \label{periodic_fn_MA}
Let $m\in\mN$, $n\in\mN>1$, and $f:\mN\to\mR_{\ge0}$ be a non-negative\lb arithmetic function.

The functions $\fnmap f$, $\fnmat f$ , and $\fnmar f$ are periodic with period $n$.
\end{lemm}
\ \\[-5ex]

\begin{proof}
Let $x\in\mZ$. By Lemma~\refs{periodic_fn} and Definition~\refs{moving_averages_f}, we get
\begin{align*}
\fnmap f(x) &= \frac{1}{m}\sum_{i=0}^{m-1}\fn f(x+i) = \frac{1}{m}\sum_{i=0}^{m-1}\fn f(x+i+n) = \fnmap f(x+n), \\
\fnmat f(x) &= \left(\prod_{i=0}^{m-1}\fn f(x+i)\right)^{\frac{1}{m}} = \left(\prod_{i=0}^{m-1}\fn f(x+i+n)\right)^{\frac{1}{m}} \\
&= \fnmat f(x+n) \text{, \hspace{70mm} and} \\
\fnmar f(x) &= \begin{cases}
	\displaystyle{\frac{m}{\sum_{i=0}^{m-1}\frac{1}{\fn f(x+i)}}} & \text{if } \fn f(x+i)>0 \text{ for all } i, \\
	0 & \text{otherwise} \end{cases} \\
&= \begin{cases}
	\displaystyle{\frac{m}{\sum_{i=0}^{m-1}\frac{1}{\fn f(x+i+n)}}} & \text{if } \fn f(x+i+n)>0 \text{ for all } i, \\
	0 & \text{otherwise} \end{cases} \\
&= \fnmar f(x+n). \qedhere
\end{align*}
\end{proof}

Analogous to Corollary~\refs{induced_pseudometric_preorder_Z_m}, the notation for the corresponding pseudometrics and preorders on $\ap \mZ m$, which refer to admissible arithmetic functions, reads as follows.

\begin{coro} \label{induced_pseudometric_preorder_SCI}
Let $x,y\in\mZ$, $m\in\mN$, $n\in\mN>1$, $f\in\msA$, $g\in\msM$, and $h\in\msI$. For $x\ne y$,\lb we set
\ \\[-4ex]
\begin{align*}
\dnSfxycp f \msA x y& = \fnmap f(x) + \fnmap f(y), \\
\dnSfxyct f \msA x y& = \fnmat f(x) + \fnmat f(y), \\
\dnSfxycr f \msA x y& = \fnmar f(x) + \fnmar f(y), \\
\dnSfxycp g \msM x y& = \fnmap g(x) + \fnmap g(y) - 2, \\
\dnSfxyct g \msM x y& = \fnmat g(x) + \fnmat g(y) - 2, \\
\dnSfxycr g \msM x y& = \fnmar g(x) + \fnmar g(y) - 2, \\
\dnSfxycp h \msI x y& = 2 - \fnmap h(x) - \fnmap h(y), \\
\dnSfxyct h \msI x y& = 2 -\fnmat h(x) - \fnmat h(y), \text{ \quad and} \\
\dnSfxycr h \msI x y& = 2 -\fnmar h(x) - \fnmar h(y),
\end{align*}
\ \\[-3ex]
whereas $d_{\dots}(\as x,\as y)=0$ \ for $x=y$. These distance functions are pseudometrics on $\ap \mZ m$. \\[-2ex]

The corresponding preorders are
\ \\[-4ex]
\begin{align*}
\as x\preceq_{\dnSfp f \msA} \as y &\iff \fnmap f(x) \le \fnmap f(y), \\
\as x\preceq_{\dnSft f \msA} \as y &\iff \fnmat f(x) \le \fnmat f(y), \\
\as x\preceq_{\dnSfr f \msA} \as y &\iff \fnmar f(x) \le \fnmar f(y), \\
\as x\preceq_{\dnSfp g \msM} \as y &\iff \fnmap g(x) \le \fnmap g(y), \\
\as x\preceq_{\dnSft g \msM} \as y &\iff \fnmat g(x) \le \fnmat g(y), \\
\as x\preceq_{\dnSfr g \msM} \as y &\iff \fnmar g(x) \le \fnmar g(y), \\
\as x\preceq_{\dnSfp h \msI} \as y &\iff \fnmap h(x) \ge \fnmap h(y), \\
\as x\preceq_{\dnSft h \msI} \as y &\iff \fnmat h(x) \ge \fnmat h(y), \text{\qquad and} \\
\as x\preceq_{\dnSfr h \msI} \as y &\iff \fnmar h(x) \ge \fnmar h(y).
\end{align*}
\end{coro}
\begin{proof}
The assertions follow from Corollary~\refs{induced_pseudometric_preorder_Z_m} because $\as x, \as y\in\ap \mZ m\subset\mZ^m$,\lb $\fn f\in\msG$, \ $\fn g\in\msQ$, \ and $\fn h\in\msT$.
\end{proof}

\pagebreak

The application of Theorem~\refs{equivalent_preorder_tuple_Z_m} to admissible arithmetic functions highlights the existence of function triples that also induce equivalent preorders on both $\mZ$ and $\mZ^m$, so also on $\ap \mZ m$.

\begin{coro} \label{admissible_function_tuple_SCI}
Let $m\in\mN$, $n\in\mN>1$, $x,y\in\mZ$, $f\in\msA$, $g\in\msM$, and $h\in\msI_0$.

For each $f\in\msA$, there exist $g\in\msM$ and $h\in\msI_0$ and vice versa, such that both\\
the preorders $\dnSf f \msA$, $\dnSf g \msM$, and $\dnSf h {\msI_0}$ are pairwise equivalent on $\mZ$, i.e.
\[ x\preceq_{\dnSf f \msA} y \iff x\preceq_{\dnSf g \msM} y \iff x\preceq_{\dnSf h {\msI_0}}y, \]
and the preorders $\dnSfp f \msA$, $\dnSft g \msM$, and $\dnSft  h {\msI_0}$ are pairwise equivalent on $\ap \mZ m$, i.e.
\[\as x\preceq_{\dnSfp f \msA} \as y \iff \as x\preceq_{\dnSft g \msM} \as y \iff \as x\preceq_{\dnSft  h {\msI_0}} \as y.\]
\end{coro}
\begin{proof}
Confer Proposition~\refs{function_triple_Z} and Theorem~\refs{equivalent_preorder_tuple_Z_m}. We have $\as x, \as y\in\ap \mZ m\subset\mZ^m$,\lb $\fn f\in\msG$, \ $\fn g\in\msQ$, \ and $\fn h\in\msT_0$.
\end{proof}
\ \\[-3ex]

Theorem~\refs{equivalent_preorder_pairs_Z_m} points to additional pairs of equivalent preorders on $\mZ^m$, so also\lb on $\ap \mZ m$. The functions $g$ and $h$ of a basic triple according to the previous corollary also induce these preorders.

\begin{coro} \label{admissible_function_pairs_SCI}
Let $m\in\mN$, $n\in\mN>1$, $x,y\in\mZ$, $g\in\msM$, and $h\in\msI_0$ with $h=\frac{1}{g}$.\lb Then
the pairs of preorders $\dnSfp g \msM$, $\dnSfr h {\msI_0}$ and $\dnSfp h {\msI_0}$, $\dnSfr g \msM$ are each equivalent\lb on $\ap \mZ m$, i.e.
\[ \as x\preceq_{\dnSfp g \msM} \as y \iff \as x\preceq_{\dnSfr h {\msI_0}} \as y \ \ \ \text{and}\ \ \ \as x\preceq_{\dnSfp h {\msI_0}} \as y \iff \seq x\preceq_{\dnSfr g \msM} \as y.\]
\end{coro}
\begin{proof}
Confer Theorem~\refs{equivalent_preorder_pairs_Z_m}. We again have $\as x, \as y\in\ap \mZ m\subset\mZ^m$, \ $\fn g\in\msM$,\lb and $\fn h\in\msI_0$.
\end{proof}
\ \\[-1ex]

We pursued pseudometrics on $\mZ^m$ and their associated preorders according to\lb Corollary~\refs{induced_pseudometric_preorder_Z_m}. Based on suitable generating functions and their supplements, we\lb examined sets of pseudometrics that induce sets of preorders on $\mZ^m$, three of which\lb coincide for appropriate function-triples. These triples exist by Theorem~\refs{equivalent_preorder_tuple_Z_m}. They also induce equivalent pseudometrics and preorders on $\mZ$ according to Proposition~\refs{function_triple_Z}. Furthermore, two functions of such a triple induce two pairs of preorders on $\mZ^m$,\lb each of which is equivalent by Theorem~\refs{equivalent_preorder_pairs_Z_m}.
\ \\[-1ex]

The application of these results to $\ap \mZ m$ and the projection onto admissible arithmetic functions was summarised in the Corollaries~\refs{induced_pseudometric_preorder_SCI}, \ \refs{admissible_function_tuple_SCI}, and \refs{admissible_function_pairs_SCI}. We conclude this section with a corresponding diagram to complement Figure~\refs{dia2} of Section~\refs{MO_Z}.

\begin{figure}[H]
  \centering\vspace*{-3mm}
\begin{gather*}\xymatrixcolsep{6pc}\xymatrixrowsep{1pc}\xymatrix{
	& \dnSfp f \msA
		\ar[r]^-{\ref{induced_pseudometric_preorder_SCI}}
	     	& \preceq_{\dnSfp f \msA}
			& \ar@<-35pt>@{<->}@/^0.75pc/[dddd]^-{\ref{admissible_function_tuple_SCI}} \\
 f\in\msA
	\ar@{<->}[ddd]_-{\ref{admissible_function_tuple_SCI}}
    	\ar[ru]^-(.61){\ref{induced_pseudometric_preorder_SCI}}
	\ar[r]^-(.5725){\ref{induced_pseudometric_preorder_SCI}}
   	\ar[rd]^-(.5425){\ref{induced_pseudometric_preorder_SCI}}
	& \dnSft f \msA
		\ar[r]^-{\ref{induced_pseudometric_preorder_SCI}}
		& \preceq_{\dnSft f \msA}
			& \\
	& \dnSfr f \msA
		\ar[r]^-{\ref{induced_pseudometric_preorder_SCI}}
	     	& \preceq_{\dnSfr f \msA}
	     		& \\
	& \dnSfp g \msM
		\ar[r]^-{\ref{induced_pseudometric_preorder_SCI}}
	     	& \preceq_{\dnSfp g \msM}
			& \ar@<-70pt>@{<->}@/^1.75pc/[ddddd]^-(.495){\ref{admissible_function_pairs_SCI}} \\
g=\e^{\ f}\in\msM
	\ar@{<->}[ddd]_-{\ref{admissible_function_tuple_SCI}}
	\ar@{<->}[ddd]^-{\ref{admissible_function_pairs_SCI}}
   	\ar[ru]^-(.545){\ref{induced_pseudometric_preorder_SCI}}
	\ar[r]^-(.475){\ref{induced_pseudometric_preorder_SCI}}
   	\ar[rd]^-(.5025){\ref{induced_pseudometric_preorder_SCI}}
	& \dnSft g \msM
		\ar[r]^-{\ref{induced_pseudometric_preorder_SCI}}
		& \preceq_{\dnSft g \msM}
			& \ar@<-35pt>@{<->}@/^0.75pc/[ddd]^-{\ref{admissible_function_tuple_SCI}} \\
	& \dnSfr g \msM
		\ar[r]^-{\ref{induced_pseudometric_preorder_SCI}}
	     	& \preceq_{\dnSfr g \msM}
			& \ar@<-62pt>@{<->}@/^0.2pc/[d]_-{\ref{admissible_function_pairs_SCI}} \\
	& \dnSfp h {\msI_0}
		\ar[r]^-{\ref{induced_pseudometric_preorder_SCI}}
	     	& \preceq_{\dnSfp h {\msI_0}}
			& \\
h=\e^{-f}\in\msI_0
   	\ar[ru]^-(.5475){\ref{induced_pseudometric_preorder_SCI}}
	\ar[r]^-(.4675){\ref{induced_pseudometric_preorder_SCI}}
   	\ar[rd]^-(.5225){\ref{induced_pseudometric_preorder_SCI}}
	& \dnSft h {\msI_0}
		\ar[r]^-{\ref{induced_pseudometric_preorder_SCI}}
		& \preceq_{\dnSft h {\msI_0}}
			& \\
	& \dnSfr h {\msI_0}
		\ar[r]^-{\ref{induced_pseudometric_preorder_SCI}}
	     	& \preceq_{\dnSfr h {\msI_0}}
			& \\
	& \dnSfp h {\msI\setminus\msI_0}
		\ar[r]^-{\ref{induced_pseudometric_preorder_SCI}}
	     	& \preceq_{\dnSfp h {\msI\setminus\msI_0}}
			& \\
h\in\msI\setminus\msI_0
   	\ar[ru]^-(.57){\ref{induced_pseudometric_preorder_SCI}}
	\ar[r]^-(.53){\ref{induced_pseudometric_preorder_SCI}}
   	\ar[rd]^-(.505){\ref{induced_pseudometric_preorder_SCI}}
	& \dnSft h {\msI\setminus\msI_0}
		\ar[r]^-{\ref{induced_pseudometric_preorder_SCI}}
	     	& \preceq_{\dnSft h {\msI\setminus\msI_0}}
			& \\
	& \dnSfr h {\msI\setminus\msI_0}
		\ar[r]^-{\ref{induced_pseudometric_preorder_SCI}}
	     	& \preceq_{\dnSfr h {\msI\setminus\msI_0}}
			& \\
}\end{gather*}
   \caption{Relationships between admissible arithmetic functions, induced\\ pseudometrics, and corresponding preorders on $\ap \mZ m$.} \label{dia4}
\end{figure}

\ \\\\


\section*{Concluding remarks}
\addcontentsline{toc}{section}{\phj Concluding remarks\phj}
\stepcounter{section}

In the preceding sections, we have investigated pseudometrics and preorders on $\mZ$, $\mZ^m$ and in particular on $\ap \mZ m$ induced by arithmetic functions. The use of arithmetic\lb functions was made possible by Lemma~\refs{AF_extend}, which describes the general case.\lb The applied transformation $x\mapsto\bigl(1+ (x-1) \modo n\bigr)$ projects $\mZ$ onto $\mZ_n\simeq\mZ/n\mZ$ and consequently $\mZ^m$ onto $\mZ_n^m$.

Our main focus was on admissible arithmetic functions, since additive and\lb multiplicative arithmetic functions are of particular interest in number theory. These represent special cases here. Many results are also applicable to the more general case.\\

Starting from the investigation of centred pseudometrics in general pseudometric spaces, Proposition~\refs{function_triple} pointed to function triples that lead to equivalent preorders on their space. A special application to admissible arithmetic functions is Proposition~\refs{function_triple_Z}, which describes a preorder-preserving isomorphism between $\msA$, $\msM$ and $\msI_0$.

These functional triples also play an important role in higher-dimensional spaces\lb as proved in Theorem~\refs{equivalent_preorder_tuple_Z_m}  and Corollary~\ref{admissible_function_tuple_SCI}\,. Due to the Pythagorean means used, they cause additional preorder-preserving isomorphisms here according to Theorem~\refs{equivalent_preorder_pairs_Z_m} and Corollary~\ref{admissible_function_pairs_SCI}\,.\\

Regarding admissible arithmetic functions, there are at most five distinct preorders on $\ap \mZ m$ induced by a corresponding function triple $f\in\msA$, $g\in\msM$, and $h\in\msI_0$,\lb as shown in Figure~\refs{dia4}. Three functions and three types of mean values result in nine preorders. There are three sets of preorders, each of whose elements are pairwise equivalent. Furthermore, two single preorder-sets still remain.
\begin{itemize}
    \item \quad $\left\{\preceq_{\dnSfp f \msA}\ ,\enspace\preceq_{\dnSft g \msM}\ ,\enspace\preceq_{\dnSft  h {\msI_0}}\right\}$, \quad cf. $
\displaystyle{\text{Corollary~\ref{admissible_function_tuple_SCI}}}$\ , \\[-0.5ex]
    \item \quad $\left\{\preceq_{\dnSfp g \msM}\ ,\enspace\preceq_{\dnSfr h {\msI_0}}\right\}$, \quad cf. $\displaystyle{\text{Corollary~\ref{admissible_function_pairs_SCI}}}$\ , \\[-0.5ex]
    \item \quad $\left\{\preceq_{\dnSfp h {\msI_0}}\ ,\enspace\preceq_{\dnSfr g \msM}\right\}$, \quad cf. $\displaystyle{\text{Corollary~\ref{admissible_function_pairs_SCI}}}$\ , \\[-0.5ex]
    \item \quad $\left\{\preceq_{\dnSft f \msA}\right\}$\ , \\[-0.5ex]
    \item \quad $\left\{\preceq_{\dnSfr f \msA}\right\}$\ .
\end{itemize}

The number of actually distinct preorders depends on the length $m$ of the sequences, the modulus $n$ and, most importantly, on the arithmetic functions under consideration. For the following discussion, we require $m,n>1$. The modulus $n=1$ was already excluded in Lemma~\refs{AF_extend}. The case $m=1$ leads to $\ap \mZ 1 \simeq \mZ$ and all regarded preorders are equivalent by Proposition~\refs{function_triple_Z}.

There is a wide variety of possible scenarios regarding the number of different\lb preorders. Smaller sequence spaces can already reach the maximum number of five distinct preorders, while larger spaces can include more equivalent preorders. Here are some examples.
\ \\[-2ex]

\begin{exam}
$n=6$, $m=2$, $\ld\in\msA$.

The Logarithmic Derivative $\ld$ \cite{Ufnarovski_Ahlander_2003, Lava_Balzarotti_2013, Ziller_2023} is the solution of the functional equation $\ld(a\cdot b)=\ld(a)+ \ld(b)$ for $a,b\in\mN$ where $\ld(1)=0$ and $\ld(p)=1/p$ for $p\in\mP$. Using the standard prime factorisation $x=\prod_{i=1}^\infty p_i^{\alpha_i}, p_i\in\mP,\alpha_i\in\mN_0$, the explicit solution is $\ld(x) = \sum_{i=1}^{\infty} \frac{\alpha_i}{p_i}$. The function is totally additive.

The functions $\e^{\ld}(x)=\e^{\ld(x)}=\e^{\sum_{i=1}^{\infty} \frac{\alpha_i}{p_i}}\in\msM$ and $\e^{-\ld}(x)=\e^{-\ld(x)}=\e^{-\sum_{i=1}^{\infty} \frac{\alpha_i}{p_i}}\in\msI_0$\lb complete the associated triple. This example includes all 5 potentially distinct\lb preorders. The basic data in the following table includes function values and the\lb corresponding moving averages.

\begin{center}
\begin{tabular}{|rr|rrrrr|}
  \hline
  \rule{0pt}{18pt }
  $x$&\small{$\ld^{\Around{\,6\,}}(x)$}&
  \small{$\ld^{\Around{\,6\,}(2,+)}_{MA} (x)$} &
  \small{$\left(\e^{\ld}\right)^{\Around{\,6\,}(2,+)}_{MA} (x)$}&
  \small{$\left(\e^{-\ld}\right)^{\Around{\,6\,}(2,+)}_{MA} (x)$}&
  \small{$\ld^{\Around{\,6\,}(2,\ldot)}_{MA} (x)$}&
  \small{$\ld^{\Around{\,6\,}(2,\rts)}_{MA} (x)$}\\[2pt]
  \hline
  \rule{0pt}{14pt}
1&0.0000&0.2500&1.3244&0.8033&0.0000&0.0000\\
2&0.5000&0.4167&1.5222&0.6615&0.4082&0.4000\\
3&0.3333&0.6667&2.0569&0.5422&0.5774&0.5000\\
4&1.0000&0.6000&1.9698&0.5933&0.4472&0.3333\\
5&0.2000&0.5167&1.7612&0.6267&0.4082&0.3226\\
6&0.8333&0.4167&1.6505&0.7173&0.0000&0.0000\\[2pt]
  \hline
\end{tabular}
\end{center}
\ \\
A few selected relations illustrate the diversity of all preorders.

\begin{center}
\begin{tabular}{|rr|ccccc|}
  \hline
  \rule{0pt}{14pt}
  &&$\as x$&$\as x$&$\as x$&$\as x$&$\as x$\\[-1pt]
  $x$&$y$&
  $\preceq_{d^{\around{\,6\,}(2,+)}_{(\ld,\,\msA)}}$&
  $\preceq_{d^{\around{\,6\,}(2,+)}_{({\e^{\ld}},\,\msM)}}$&
  $\preceq_{d^{\around{\,6\,}(2,+)}_{({\e^{-\ld}},\,{\msI_0})}}$&
  $\preceq_{d^{\around{\,6\,}(2,\ldot)}_{(\ld,\,\msA)}}$&
  $\preceq_{d^{\around{\,6\,}(2,\rts)}_{(\ld,\,\msA)}}$\\[2pt]
  &&$\as y$&$\as y$&$\as y$&$\as y$&$\as y$\\[3pt]
  \hline
  \rule{0pt}{14pt}
6&2&\cp{\ true}&\cl{ false}&\cw{\ true}&\cw{\ true}&\cw{\ true}\\
2&6&\cp{\ true}&\cp{\ true}&\cl{ false}&\cw{ false}&\cw{ false}\\
5&2&\cp{ false}&\cp{ false}&\cp{ false}&\cl{\ true}&\cw{\ true}\\
4&2&\cp{ false}&\cp{ false}&\cp{ false}&\cp{ false}&\cl{\ true}\\[2pt]
  \hline
\end{tabular}
\end{center}
\end{exam}
\ \\[-3ex]

\begin{exam} \label{exampleA}
$n=13$ , $m=2$, $\nd\in\msM$.

The function $\nd$ is the number of divisors of $x$. It  is multiplicative. The associated functions are $\log(\nd(x))\in\msA$ and $\frac{1}{\nd(x)}\in\msI_0$.

The following preorders are also equivalent here. For $x,y\in\mZ$, we have
\ \\[-1ex]
\[ \as x \preceq_{d^{\around{13}(2,+)}_{({\frac{1}{\nd(x)}},\,{\msI_0})}} \as y \quad\Longleftrightarrow\quad 
  \as x \preceq_{d^{\around{13}(2,\ldot)}_{(\log(\nd(x)),\,\msA)}} \as y \quad\Longleftrightarrow\quad 
  \as x \preceq_{d^{\around{13}(2,\rts)}_{(\log(\nd(x)),\,\msA)}} \as y. \]
\ \\[-1ex]
Finally, there are three distinct preorders left.
 
\begin{center}
\begin{tabular}{|rr|ccccc|}
  \hline
  \rule{0pt}{14pt}
  &&$\as x$&$\as x$&$\as x$&$\as x$&$\as x$\\[-1pt]
  $x$&$y$&
  $\preceq_{d^{\around{13}(2,+)}_{(\log(\nd(x)),\,\msA)}}$&
  $\preceq_{d^{\around{13}(2,+)}_{(\nd,\,\msM)}}$&
  $\preceq_{d^{\around{13}(2,+)}_{({\frac{1}{\nd(x)}},\,{\msI_0})}}$&
  $\preceq_{d^{\around{13}(2,\ldot)}_{(\log(\nd(x)),\,\msA)}}$&
  $\preceq_{d^{\around{13}(2,\rts)}_{(\log(\nd(x)),\,\msA)}}$\\[3pt]
  &&$\as y$&$\as y$&$\as y$&$\as y$&$\as y$\\[3pt]
  \hline
  \rule{0pt}{14pt}
11& 8&\cp{\ true}&\cl{ false}&\cw{\ true}&\cw{\ true}&\cw{\ true}\\
 8&11&\cp{\ true}&\cp{\ true}&\cl{ false}&\cw{ false}&\cw{ false}\\[2pt]
  \hline
\end{tabular}
\end{center}
\end{exam}
\ \\[-2.5ex]

\begin{exam} \label{exampleB}
$n=9$ , $m=4$, $\omega\in\msA$.

The number of distinct prime divisors $\omega(x)$ is additive. The associated functions are $\e^{\omega(x)}\in\msM$ and $\e^{-\omega(x)}\in\msI_0$. In this case, all related preorders coincide.
\end{exam}
\ \\[-2.5ex]

There are a few general statements we can make about the number of distinct\lb preorders if $m$ remains fixed and $n$ increases. \\

The crucial interval here is $1\le z\le n-m+1$. Potential differences between\lb preorders in this area persist even with increasing $n$. With a larger $n$, even more\lb different preorders are possible. All of this applies, among other things, to Example~\refs{exampleA}.

If there were no differences between the preorders in the range $1\le z\le n-m+1$,\lb i.e., if all preorders were equivalent, this also would remain the case even with\lb increasing $n$. However, values in later positions can again cause other relationships.

\begin{lemm} \label{distinct_order}
Let $n\in\mN>1$, $m\in\mN\le n$, $f\in\msA$, $g\in\msM$, and $h\in\msI_0$. Furthermore, let $f$, $g$, and $h$ form a triple according to Proposition~\refs{function_triple_Z}, i.e. $f\in\msA$, \ $g=\e^{\ f}\in\msM$ and $h=\e^{-f}\in\msI_0$.

If the two preorders 
$\preceq_{\dd n 1}$ and $\preceq_{\dd n 2}$,\\
\hspace*{10mm} where $(\mathpzc{f}_1,\mathcal{S}_1),(\mathpzc{f}_2,\mathcal{S}_2)\in\{(f,\msA),(g,\msM),(h,\msI_0)\}$ and $\mathcal{O}_1,\mathcal{O}_2\in\{+,\ldot,\rts\}$,\\
are not equivalent on $\ap \mZ m$ at two positions $x,y\in\mN\le n-m+1$,\\
\hspace*{10mm} i.e. $\as x \preceq_{\dd n 1} \as y$ and not $\as x \preceq_{\dd n 2} \as y$,\\ 
then the preorders  $\preceq_{\dd k 1}$ and $\preceq_{\dd k 2}$ are not equivalent on $\ap \mZ m$ for all $k\ge n$.
\end{lemm}
\begin{proof}
In Corollary~\refs{induced_pseudometric_preorder_SCI}, relations of the type $\as x \preceq_{\dd n 1} \as y$ were defined as\lb inequalities between the values of the corresponding moving average $\ff n 1$ of the function $\fn {\mathpzc{f}_1}$. For each $z\in\mN\le n-m+1$, the moving average $\ff n 1(z)$\lb depends by Definition~\refs{moving_averages_f} only on the function values $\fn {\mathpzc{f}_1} (i)$ for $i=z,\dots, z+m-1$,\lb\\[-2ex]
i.e. $i\le z+m-1\le n-m+1+m-1=n$.\\[-1ex]

Then $\ff n 1(z)=\ff k 1(z)$ for all $k\ge n$ because\\[-1ex]
\[\fn {\mathpzc{f}_1} (i)=\mathpzc{f}_1\bigl(1+ (i-1) \modo n\bigr)=\mathpzc{f}_1 (i)=
\mathpzc{f}_1\bigl(1+ (i-1) \modo k\bigr)=\mathpzc{f}_1^{\around{k}} (i)\]
according to Lemma~\refs{AF_extend}.\\[-2ex]
 
The same applies analogously to $\as x \preceq_{\dd n 2} \as y$.
\end{proof}

For monotonic functions, the corresponding moving averages are also monotonic. This results in a simplification. So, there can be no differences between the preorders within the specified interval even for increasing $n$. This turns out to be a special case of Lemma~\refs{distinct_order} for equivalent preorders.

The function $\omega(z)$ in Example~\refs{exampleB} is not monotonic in general. Nevertheless, the moving averages are monotonic for $1\le z\le 9-4+1=6$. This is sufficient to draw the same conclusions here.

\begin{lemm} \label{monotonic_function}
Let $n\in\mN>1$, $m\in\mN\le n$, $f\in\msA$, $g\in\msM$, and $h\in\msI_0$. Furthermore, let $f$, $g$, and $h$ form a triple according to Proposition~\refs{function_triple_Z}, i.e. $f\in\msA$, \ $g=\e^{\ f}\in\msM$ and $h=\e^{-f}\in\msI_0$.

If the functions $f$, $g$, and $h$ are monotonic, then the preorders 
$\preceq_{\dd k 1}$ and $\preceq_{\dd k 2}$,\\
\hspace*{10mm} where $(\mathpzc{f}_1,\mathcal{S}_1),(\mathpzc{f}_2,\mathcal{S}_2)\in\{(f,\msA),(g,\msM),(h,\msI_0)\}$ and $\mathcal{O}_1,\mathcal{O}_2\in\{+,\ldot,\rts\}$,\\
are equivalent on $\ap \mZ m$ at arbitrary positions $x,y\in\mN\le n-m+1$,\\
\hspace*{10mm} i.e. $\as x \preceq_{\dd k 1} \as y \Longleftrightarrow \as x \preceq_{\dd k 2} \as y$,\\
for all $k\ge n$.
\end{lemm}
\begin{proof}
The transformations $g=\e^{\ f}$ and $h=\e^{-f}$ are bijective because $\e^x$ is strictly monotonic. If one of these functions is monotonic, then all other functions are also monotonic. Without loss of generality, we assume that the functions are not constant. Then $f$ and $g$ must be monotonically increasing, and $h$ is monotonically decreasing. Therefore, the corresponding moving averages are also monotonic. \\[-1ex]

As proven above, $\ff n 1(z)=\ff k 1(z)$ for all $k\ge n$ and  $z\in\mN\le n-m+1$.\lb \\[-2ex]
By exploiting the monotonicity, we obtain \\[-3ex]
\begin{align*}
\as x \preceq_{\dd k 1} \as y &\Longleftrightarrow \ff k 1(x) \le \ff k 1(y)\\[-1ex]
&\Longleftrightarrow \ff k 2(x) \le \ff k 2(y) \Longleftrightarrow \as x \preceq_{\dd k 2} \as y. \\[-4ex]
\end{align*}
for $\mathpzc{f}_1,\mathpzc{f}_2\in\{\msA\cup\msM\}$. If $\mathpzc{f}_1\in\msI_0$ or $\mathpzc{f}_2\in\msI_0$, the corresponding $\le$ changes to $\ge$.
\end{proof}
\ \\[-2ex]

We would like to draw attention to another important topic related to the study of integer sequences, which could be investigated using the pseudometrics and\lb preorders presented in this paper: the search for sequences with extreme properties that can be expressed as function values.

The periodicity of the moving averages established in Lemma~\refs{periodic_fn_MA} means that\lb the number of existing values is finite. At most $n$ values of these can be different. Therefore, extreme values also exist for the distances of the pseudometrics defined in Corollary~\refs{induced_pseudometric_preorder_SCI} to the centre $(1)\in\mZ^m$.

In simple cases, general statements can be made about extreme values of moving averages. We build on the last lemma and examine monotonic functions as a first example. It is not limited to admissible functions. The following results apply to all generating and supplementary functions.

\begin{lemm} \label{monotonic_extreme}
Let $n\in\mN>1$, $m\in\mN\le n$, $f\in\msG$, $g\in\msQ$, and $h\in\msT_0$. If the functions $f$, $g$, and $h$ are monotonic, then
\begin{align*}
\max_{{x\in\mZ_n}}\ f^{\Around{n}(m,\mathcal{O}_1)}_{MA}(x)&=f^{\Around{n}(m,\mathcal{O}_1)}_{MA}(n-m+1), \\
\min_{{x\in\mZ_n}}\ f^{\Around{n}(m,\mathcal{O}_1)}_{MA}(x)&=f^{\Around{n}(m,\mathcal{O}_1)}_{MA}(1), \\
\max_{{x\in\mZ_n}}\ g^{\Around{n}(m,\mathcal{O}_2)}_{MA}(x)&=g^{\Around{n}(m,\mathcal{O}_2)}_{MA}(n-m+1), \\
\min_{{x\in\mZ_n}}\ g^{\Around{n}(m,\mathcal{O}_2)}_{MA}(x)&=g^{\Around{n}(m,\mathcal{O}_2)}_{MA}(1), \\
\max_{{x\in\mZ_n}}\ h^{\Around{n}(m,\mathcal{O}_3)}_{MA}(x)&=h^{\Around{n}(m,\mathcal{O}_3)}_{MA}(1), \text{\qquad and} \\
\min_{{x\in\mZ_n}}\ h^{\Around{n}(m,\mathcal{O}_3)}_{MA}(x)&=h^{\Around{n}(m,\mathcal{O}_3)}_{MA}(n-m+1),
\end{align*}
where $\mathcal{O}_1,\mathcal{O}_2,\mathcal{O}_3\in\{+,\ldot,\rts\}$. 
\end{lemm}
\begin{proof}
Without loss of generality, we again assume non-constant functions. Then $f$ and $g$ are monotonically increasing, and $h$ is monotonically decreasing. We prove the assertion for $f\in\msG$ and $\mathcal{O}_1=+$. The other assertions follow analogously. 

The value $\fnmap f(1) = \displaystyle{\frac{1}{m}\sum_{i=1}^{m}\fn f(i)}$ must be minimum because $f$ is monotonically increasing. The assertion about the maximum follows by contradiction.

We assume there were $y\in\mZ_n$ with $\fnmap f(y)>\fnmap f(n-m+1)$. According to the assumptions, $1\le n-m+1\le n$. The function $\fn f$ is periodic by Lemma~\refs{periodic_fn}. Therefore, we can distinguish two cases for $y$. These are $1\le y<n-m+1$ and $n-m+1<y\le n$. In the first case, we get
\[ \hspace*{19.7mm}\fnmap f(y) = \frac{1}{m}\sum_{i=0}^{m-1}\fn f(y+i) > \frac{1}{m}\sum_{i=0}^{m-1}\fn f(n-m+1+i), \]
in the second case
\begin{align*}
\hspace*{21mm}\fnmap f(y) &= \frac{1}{m}\sum_{i=0}^{m-1}\fn f(y+i) \\
&= \frac{1}{m}\left( \sum_{i=y}^{n}\fn f(i) + \sum_{i=n+1}^{y+m-1}\fn f(i) \right) \\
&= \frac{1}{m}\sum_{i=y}^{n}\fn f(i) + \frac{1}{m}\sum_{i=0}^{m-1-(n-y+1)}\fn f(1+i) \\
&> \frac{1}{m}\sum_{i=0}^{m-1}\fn f(n-m+1+i) \qquad\text{(by assumption)}
\end{align*}
\begin{align*}
&= \frac{1}{m}\sum_{i=n-m+1}^{n}\fn f(i) \\
&= \frac{1}{m}\left( \sum_{i=n-m+1}^{y-1}\fn f(i) + \sum_{i=y}^{n}\fn f(i) \right) \\
&= \frac{1}{m}\sum_{i=0}^{m-1-(n-y+1)}\fn f(n-m+1+i) + \frac{1}{m}\sum_{i=y}^{n}\fn f(i), \\
\frac{1}{m}\sum_{i=0}^{m-1-(n-y+1)}\fn f(1+i) &> \frac{1}{m}\sum_{i=0}^{m-1-(n-y+1)}\fn f(n-m+1+i).
\end{align*}
Both results contradict the monotonicity of $f$.
\end{proof}
\ \\[-3ex]

\begin{exam} \label{exampleC}
$n=17$ , $m=5$, $\frac{1}{1+\theta}\in\msT_0$.\ \\[-2ex]

The Chebyshev function $\theta(x) =\sum_{p\in\mP\le x} \log(p)$ is neither additive nor multi-\lb plicative. Due to $\theta(1)=0$, $h=\frac{1}{1+\theta}\in\msT_0$. The corresponding minima are reached\lb at $x=17-5+1=13$. We present the results for $h^{\Around{\,17\,}}$: \\
\begin{center}
\begin{tabular}{|rr|rrr|}
  \hline
  \rule{0pt}{18pt}
  $x$&\small{$h^{\Around{\,17\,}}(x)$}&
  \small{$h^{\Around{\,17\,}(5,+)}_{MA} (x)$}&
  \small{$h^{\Around{\,17\,}(5,\ldot)}_{MA} (x)$}&
  \small{$h^{\Around{\,17\,}(5,\rts)}_{MA} (x)$}\\[2pt]
  \hline
  \rule{0pt}{14pt}
 1&1.0000&\cl{0.5068}&\cl{0.4438}&\cl{0.3944}\\
 2&0.5906&\cw{0.3523}&\cw{0.3300}&\cw{0.3110}\\
 3&0.3582&\cw{0.2657}&\cw{0.2533}&\cw{0.2412}\\
 4&0.3582&\cw{0.2255}&\cw{0.2150}&\cw{0.2059}\\
 5&0.2272&\cw{0.1854}&\cw{0.1824}&\cw{0.1796}\\
 6&0.2272&\cw{0.1715}&\cw{0.1695}&\cw{0.1678}\\
 7&0.1576&\cw{0.1489}&\cw{0.1478}&\cw{0.1465}\\
 8&0.1576&\cw{0.1403}&\cw{0.1386}&\cw{0.1369}\\
 9&0.1576&\cw{0.1264}&\cw{0.1235}&\cw{0.1205}\\
10&0.1576&\cw{0.1126}&\cw{0.1100}&\cw{0.1076}\\
11&0.1144&\cw{0.0988}&\cw{0.0980}&\cw{0.0972}\\
12&0.1144&\cw{0.0936}&\cw{0.0931}&\cw{0.0926}\\
13&0.0884&\cp{0.0849}&\cp{0.0846}&\cp{0.0842}\\
14&0.0884&\cw{0.2672}&\cw{0.1373}&\cw{0.1019}\\
15&0.0884&\cw{0.3676}&\cw{0.2008}&\cw{0.1267}\\
16&0.0884&\cw{0.4216}&\cw{0.2656}&\cw{0.1616}\\
17&0.0707&\cw{0.4755}&\cw{0.3514}&\cw{0.2230}\\[2pt]
  \hline
\end{tabular}
\end{center}
\end{exam}
\pagebreak

This section is rounded off by the description of a general duality between minima and maxima of moving averages across different sequent lengths. It applies to all\lb non-negative arithmetic functions and also utilizes the periodicity of its extensions\lb to $\mZ$ according to Lemma~\refs{AF_extend}.
\ \\[-1ex]

\begin{lemm} \label{duality_extreme}
Let $n\in\mN>1$, $m\in\mN<n$, and $f:\mN\to\mR_{\ge0}$. \\
Furthermore, we define $x_0=\displaystyle{\argmax_{x\in\mZ_n}}\ \fnmap f(x)$ and $x_1=\displaystyle{\argmin_{x\in\mZ_n}}\ \fnmap f(x)$.

Then
\begin{align*} 
f^{\around{n}(n-m,+)}_{MA}(x_0+m) & =\displaystyle{\min_{x\in\mZ}}\ f^{\around{n}(n-m,+)}_{MA}(x) \quad\text{ and vice versa} \\
f^{\around{n}(n-m,+)}_{MA}(x_1+m) & =\displaystyle{\max_{x\in\mZ}}\ f^{\around{n}(n-m,+)}_{MA}(x).
\end{align*}
\end{lemm}
\begin{proof}
By assumption and Lemma~\refs{periodic_fn}, $\fn f$ is periodic with period $n$ and therefore
\[ f^{\around{n}(n,+)}_{MA}(x)=\displaystyle{\frac{1}{n}\sum_{i=0}^{n-1}}\fn f(x+i)
=\frac{1}{n}\displaystyle{\sum_{i=1}^{n}}\fn f(i)=const\underset{Df}{=}y. \]

For all $x\in\mZ$, we have
\begin{align*} 
f^{\around{n}(n-m,+)}_{MA}(x+m) &= \frac{1}{n}\sum_{i=0}^{n-m-1} \fn f(x+m+i) \\
&= \frac{1}{n}\sum_{i=0}^{n-1} \fn f(x+m+i) - \frac{1}{n}\sum_{i=n-m}^{n-1} \fn f(x+m+i) \\
&= y - \frac{1}{n}\sum_{i=n-m-n+m}^{n-1-n+m} \fn f(x+m+i+n-m) \\
&= y - \frac{m}{n}\mdot\frac{1}{m}\sum_{i=0}^{m-1} \fn f(x+n+i) = y - \frac{m}{n}\mdot\frac{1}{m}\sum_{i=0}^{m-1} \fn f(x+i) \\
&= y - \frac{m}{n}\mdot\fnmap f(x).
\end{align*}
Thus, $f^{\around{n}(n-m,+)}_{MA}(x+m)$ is minimum if $\fnmap f(x)$ is maximum and vice versa.
\end{proof}
\ \\[-4ex]
\begin{rema}
This lemma is even more general. For positive functions, this applies\lb accordingly to geometric and harmonic moving averages.

For monotonic functions, it is a consequence of the previous lemma. The same result would be obtained if $m$ in Lemma~\refs{monotonic_extreme} were replaced by $n-m$.
\end{rema}
\pagebreak

\begin{exam}
$n=8$ , $m=3$, $\varphi\in\msM$, $n-m=8-3=5$. \\[-2ex]

Euler's totient function $\varphi$ counts the natural numbers up to a given integer $n$ that are coprime to $n$. In this case, there are two maxima at $x_{0_1}=5$ and $x_{0_2}=6$ and one minimum at $x_1=1$. Therefore,
\begin{align*} 
x_{0_1}+m&\quad=\quad5+3\hspace{11.7mm}=\quad 8, \\
x_{0_2}+m&\quad=\quad6+3=9 \quad\congr\quad \,1 \modu {8}, \\
x_1+m&\quad=\quad1+3\hspace{11.7mm}=\quad 4.
\end{align*}
The following table lists the corresponding results.
 \ \\[-1ex]
 
\begin{center}
\begingroup
\setlength{\tabcolsep}{5.5pt}
\begin{tabular}{|rr|rrr|rrr|}
  \hline
  \rule{0pt}{18pt}$x$&\small{$\varphi^{\Around{\,8\,}}(x)$}&
  \small{$\varphi^{\Around{\,8\,}(3,+)}_{MA} (x)$}\ &
  \small{$\varphi^{\Around{\,8\,}(3,\ldot)}_{MA} (x)$}\ &
  \small{$\varphi^{\Around{\,8\,}(3,\rts)}_{MA} (x)$}\ &
  \small{$\varphi^{\Around{\,8\,}(5,+)}_{MA} (x)$}\ &
  \small{$\varphi^{\Around{\,8\,}(5,\ldot)}_{MA} (x)$}\ &
  \small{$\varphi^{\Around{\,8\,}(5,\rts)}_{MA} (x)$}\\[2pt]
  \hline
  \rule{0pt}{14pt}
1&1&\cp{1.3333}&\cp{1.2599}&\cp{1.2000}&\cp{2.0000}&\cp{1.7411}&\cp{1.5385}\\
2&1&\cw{1.6667}&\cw{1.5874}&\cw{1.5000}&\cw{2.2000}&\cw{2.0000}&\cw{1.8182}\\
3&2&\cw{2.6667}&\cw{2.5198}&\cw{2.4000}&\cw{3.2000}&\cw{2.8619}&\cw{2.6087}\\
4&2&\cw{2.6667}&\cw{2.5198}&\cw{2.4000}&\cl{3.6000}&\cl{3.2875}&\cl{3.0000}\\
5&4&\cl{4.0000}&\cl{3.6342}&\cl{3.2727}&\cw{3.4000}&\cw{2.8619}&\cw{2.3077}\\
6&2&\cl{4.0000}&\cl{3.6342}&\cl{3.2727}&\cw{2.8000}&\cw{2.1689}&\cw{1.7143}\\
7&6&\cw{3.6667}&\cw{2.8845}&\cw{2.1176}&\cw{2.8000}&\cw{2.1689}&\cw{1.7143}\\
8&4&\cw{2.0000}&\cw{1.5874}&\cw{1.3333}&\cp{2.0000}&\cp{1.7411}&\cp{1.5385}\\[2pt]
  \hline
\end{tabular}
\endgroup
\end{center}
\end{exam}

\vspace*{8.5cm}


\subsubsection*{Contact}
marioziller@arcor.de
\pagebreak
\bibliography{References}     
\addcontentsline{toc}{section}{\phj References\phj}
\end{document}